\documentclass[11pt]{amsart}

\usepackage{amsmath,amssymb,mathtools}
\usepackage[hidelinks]{hyperref}

\numberwithin{equation}{section}
\allowdisplaybreaks[4]

\newtheorem{theorem}{Theorem}[section]
\newtheorem{lemma}[theorem]{Lemma}
\newtheorem{proposition}[theorem]{Proposition}
\newtheorem{corollary}[theorem]{Corollary}
\theoremstyle{remark}
\newtheorem{remark}[theorem]{Remark}

\title[Liouville rigidity for a Lorentzian mean-curvature equation]
{Liouville Rigidity and Universal Spacelikeness Estimates for a Lorentzian Prescribed Mean Curvature Equation}

\author{Xi-Nan Ma}
\address{School of Mathematical Sciences, University of Science and Technology of China, Hefei, Anhui 230026, People's Republic of China}
\email{xinan@ustc.edu.cn}

\author{Tian Wu}
\address{School of Mathematical Sciences, University of Science and Technology of China, Hefei, Anhui 230026, People's Republic of China}
\email{wt1997@ustc.edu.cn}

\author{Wangzhe Wu}
\address{School of Mathematics and Statistics, Ningbo University, Ningbo, Zhejiang, People's Republic of China}
\email{wuwz18@mail.ustc.edu.cn}

\author{Bao Yu}
\address{School of Mathematical Sciences, University of Science and Technology of China, Hefei, Anhui 230026, People's Republic of China}
\email{baoyu1@mail.ustc.edu.cn}

\date{}

\begin{document}

\begin{abstract}
We prove a Liouville theorem for nonnegative entire strictly spacelike solutions of
\[
\operatorname{div}\left(\frac{\nabla u}{\sqrt{1-|\nabla u|^2}}\right)+u^p=0
\qquad\text{in }\mathbb R^n.
\]
If $n=2$ and $p\geqslant1$, or if $n\geqslant3$ and
$1\leqslant p\leqslant\frac{n+2}{n-2}$, every nonnegative $C^2$ solution
satisfying $|\nabla u|<1$ vanishes identically. This resolves, in the
classical strictly spacelike setting, the nonexistence conjecture of Byeon,
Ikoma, Malchiodi, and Mari \cite{ByeonIkomaMalchiodiMari2024},
including the critical endpoint. No symmetry, decay,
integrability, or uniform spacelike gap is assumed. A key ingredient is a universal
bound, valid for every $n\geqslant2$ and $p\geqslant1$, for both the height $u$
and the Lorentz factor $(1-|\nabla u|^2)^{-1/2}$. Then a weighted trace-free
tensor identity from the invariant-tensor
approach, combined with a common cutoff estimate, a core-counting
argument and Souplet-type feedback inequality, yields a unified proof in the subcritical and critical ranges.
The upper endpoint is sharp for $n\geqslant3$, as supercritical radial
solutions exist.
The theorem also gives half-space rigidity for complete spacelike
hypersurfaces, including at the critical exponent.
\end{abstract}


\maketitle

\section{Introduction}\label{sec:introduction}

We consider the quasilinear equation
\begin{equation}\label{eq:main-equation}
\operatorname{div}\left(\frac{\nabla u}{\sqrt{1-|\nabla u|^2}}\right)+u^p=0
\qquad\text{in }\mathbb R^n,
\end{equation}
under the spacelike constraint $|\nabla u|<1$. The graph of $u$ is a
spacelike hypersurface in Lorentz--Minkowski space, and the differential
operator in \eqref{eq:main-equation} is the Lorentzian mean-curvature
operator of this graph. The geometric theory includes the Bernstein-type
rigidity theorem of Cheng and Yau \cite{ChengYau1976}, the Dirichlet theory
of Bartnik and Simon \cite{BartnikSimon1982}, and Treibergs' construction of
entire constant-mean-curvature hypersurfaces \cite{Treibergs1982}. The same
operator arises as the electrostatic Born--Infeld operator introduced by Born
and Infeld \cite{BornInfeld1934}. Its modern variational and regularity theory
has been developed, among others, by Bonheure, d'Avenia, and Pomponio
\cite{BonheureDaveniaPomponio2016} and by Byeon, Ikoma, Malchiodi, and Mari
\cite{ByeonIkomaMalchiodiMari2024}. In both settings, the principal analytic
difficulty is the singular behavior of the operator as $|\nabla u|\uparrow1$.

The entire-space problem for \eqref{eq:main-equation} has previously been
studied mainly through radial reduction. Bonheure, De Coster, and Derlet
obtained radial solutions for Lorentz--Minkowski mean-curvature equations
\cite{BonheureDeCosterDerlet2012}, while Azzollini developed a radial
ground-state theory \cite{Azzollini2014,Azzollini2016}. For the pure-power
equation \eqref{eq:main-equation} in dimensions $n\geqslant3$, Azzollini
proved nonexistence of radial ground states for
$1<p<(n+2)/(n-2)$ and existence for $p>(n+2)/(n-2)$
\cite[Theorems~0.1 and~0.2]{Azzollini2016}. That work left the nonradial
problem open in the absence of a symmetry theorem.

Byeon, Ikoma, Malchiodi, and Mari subsequently proved nonexistence under a
finite-energy assumption \cite[Proposition~2.7]{ByeonIkomaMalchiodiMari2026}
and conjectured that no positive weak solutions exist in either the subcritical
or critical range without additional restrictions
\cite[Remark~2.8]{ByeonIkomaMalchiodiMari2026}. In the notation of
\eqref{eq:main-equation}, the conjectured range is
$1<p\leqslant(n+2)/(n-2)$ for $n\geqslant3$. Our main theorem resolves
this conjecture in the classical strictly spacelike setting, without any
symmetry, decay, integrability, or uniform spacelike gap assumption.

From the Liouville perspective, \eqref{eq:main-equation} should be compared
with the Lane--Emden equation $\Delta u+u^p=0$ in $\mathbb R^n$. For the
semilinear model, subcritical nonexistence, radial symmetry, and critical
classification are classical; see Gidas and Spruck \cite{GidasSpruck1981},
Gidas, Ni, and Nirenberg \cite{GidasNiNirenberg1979}, and Caffarelli, Gidas,
and Spruck \cite{CaffarelliGidasSpruck1989}. General methods for quasilinear
Liouville theorems and universal bounds were developed by Serrin and Zou
\cite{SerrinZou2002}. Equation \eqref{eq:main-equation}, however, is neither
scale invariant nor uniformly elliptic under the pointwise spacelike
condition alone. Uniform spacelikeness and control of the operator under
rescaling are therefore needed before the Euclidean classification can be used.

Our main Liouville theorem is the following.

\begin{theorem}\label{thm:main}
Suppose that $n\geqslant2$ and that
\[
\begin{cases}
p\geqslant1, & n=2,\\[2mm]
1\leqslant p\leqslant\dfrac{n+2}{n-2}, & n\geqslant3.
\end{cases}
\]
Let $u\in C^2(\mathbb R^n)$ be a nonnegative entire solution of
\eqref{eq:main-equation} satisfying $|\nabla u|<1$ in $\mathbb R^n$. Then
$u\equiv0$.
\end{theorem}

A second main result supplies the uniform ellipticity that is not assumed in
Theorem~\ref{thm:main}. These estimates hold for every $p\geqslant1$,
independently of the Sobolev restriction.

\begin{theorem}[Universal height and spacelikeness estimates]
\label{thm:universal-bounds}
Let $n\geqslant2$ and $p\geqslant1$. There exists $C=C(n,p)>1$ such that every
nonnegative entire $C^2$ solution of \eqref{eq:main-equation} satisfying
$|\nabla u|<1$ obeys
\[
\sup_{\mathbb R^n}u\leqslant C,
\qquad
\sup_{\mathbb R^n}\frac{1}{\sqrt{1-|\nabla u|^2}}\leqslant C.
\]
In particular,
\[
|\nabla u|\leqslant\sqrt{1-C^{-2}}<1
\qquad\text{throughout }\mathbb R^n.
\]
\end{theorem}

Theorem~\ref{thm:main} includes the Sobolev endpoint and assumes no
symmetry, decay, finite energy, or uniform spacelike gap. At the critical
exponent, the Lane--Emden equation admits positive entire bubbles
\cite{CaffarelliGidasSpruck1989}, whereas the Lorentzian equation
\eqref{eq:main-equation} admits only the zero solution. The supercritical
radial solutions in \cite[Theorem~0.2]{Azzollini2016} show that the upper
endpoint $(n+2)/(n-2)$ is sharp for $n\geqslant3$.

The analytic theorem has the following coordinate-free geometric
consequence. Scalar mean curvature below denotes the trace, rather than the
normalized trace, of the second fundamental form.

\begin{corollary}[Geometric half-space rigidity]\label{cor:geometric-halfspace}
Let $n$ and $p$ satisfy the hypotheses of Theorem~\ref{thm:main}. Let
$P\subset\mathbb R^{n,1}$ be a spacelike affine hyperplane, let $T$ be its
future-directed unit timelike normal, and fix $X_0\in P$. Suppose that
$F:\Sigma^n\to\mathbb R^{n,1}$ is a connected, complete, $C^2$ spacelike
immersion. Define its time-height relative to $P$ by
$\tau\coloneqq-\langle F-X_0,T\rangle_{\mathbb R^{n,1}}$, and assume that
$\tau\geqslant0$. If the scalar mean curvature with respect to the
future-directed unit normal satisfies $H=-\tau^p$, then $F(\Sigma)=P$.
\end{corollary}

For compactly supported variations, equation \eqref{eq:main-equation} is the
Euler--Lagrange equation of
\[
\mathcal J_p(u;\Omega)
\coloneqq\int_\Omega\left(\sqrt{1-|\nabla u|^2}
+\frac{|u|^{p+1}}{p+1}\right)\,\mathrm dx.
\]
Thus Corollary~\ref{cor:geometric-halfspace} can also be viewed as a rigidity
statement for complete spacelike hypersurfaces that are stationary for this
functional and lie on one side of a reference hyperplane.

We briefly describe the proof. Writing
$W\coloneqq(1-|\nabla u|^2)^{-1/2}$, we first establish the pointwise Bernstein
estimate
\[
W\leqslant C(1+u)^{p+1}.
\]
The argument uses the induced metric on the spacelike graph. Its Lorentzian
cutoff is adapted to both the horizontal and vertical variables. A comparison
with an explicit hyperbolic cap then yields a universal height bound. Together
these estimates prove Theorem~\ref{thm:universal-bounds} and make
\eqref{eq:main-equation} uniformly elliptic.

The cases $p=1$ and $n=2$ follow from a logarithmic test. In the remaining
range, we use the invariant-tensor
approach to get one differential identity; see Obata \cite{Obata1962} and
\cite{MaWu2026,MaWuInvariant2024}.

To close the argument, we combine a core-counting idea from
\cite{LiLiuMa2026} with the integral-feedback method of \cite{Souplet2009}. In order to control $\mathfrak D(R)$ in \eqref{eq:liouville-integrals}, we need to  improve the integral estimates $\mathcal U(R)$ in \eqref{eq:liouville-U} for \(u\). So we cover the region where
\(u\) is large by cores at its natural scale
\(\ell=u^{-(p-1)/2}\). The main task is then to control how many such
cores can occur. For this purpose, we show that a relevant core cannot
have arbitrarily small normalized defect \eqref{09221246}. Indeed, if a sequence of
large-\(u\) cores had defect \eqref{09221246} tending to zero, normalization at their natural
scales would produce a Euclidean Lane--Emden equation limit with value one at the
origin. The subcritical case is ruled out by the Gidas--Spruck Liouville
theorem. At the critical exponent the limit is a standard bubble, and a
Green-kernel and Pohozaev argument rules out a persistent neck around it.
Thus we get a lower bound for normalized defect \eqref{09221246} and the cores covering the large-\(u\) region can ultimately be controlled
by the total defect $\mathfrak D(C_0R)$ where $C_0>1$ is a constant.

The core-counting estimate improves the integral control of \(u\), which in turn yields a Souplet-type feedback inequality for the defect: using $\mathfrak D(C_0R)$ to control $\mathfrak D(R)$. Following Souplet's radius-selection argument \cite{Souplet2009}, we close this feedback and conclude that the defect vanishes identically. The rigidity estimate then forces \(u\equiv0\).

The paper is organized as follows. Section~\ref{sec:geometric} records the
geometric identities and the regularity reduction. Section~\ref{sec:Bernstein}
establishes the Bernstein estimate, and Section~\ref{sec:bound} proves the
height bound and Theorem~\ref{thm:universal-bounds}.
Section~\ref{sec:identities} records the notation and the weighted
differential identity from the invariant-tensor
approach. Section~\ref{sec:liouville} proves
Theorem~\ref{thm:main} in the subcritical and critical ranges by  a core-counting and Souplet-type feedback arguments, and establishes
Corollary~\ref{cor:geometric-halfspace}.

For completeness, the elementary one-dimensional analogue is recorded in
Remark~\ref{rem:one-dimensional}.

\section{Geometric preliminaries and regularity}\label{sec:geometric}

Let $n\geqslant2$ and $p\geqslant1$, and let $u$ be a nonnegative entire
$C^2$ solution of
\[
\operatorname{div}(W\nabla u)+u^p=0,
\qquad
W\coloneqq(1-|\nabla u|^2)^{-\frac12}.
\]
We assume that $u$ is strictly spacelike; that is,
$|\nabla u(x)|<1$ for every $x\in\mathbb R^n$.

\begin{lemma}[Regularity reduction]\label{lem:regularity-reduction}
Suppose that $n\geqslant2$ and $p \geqslant 1$.
Let $u\in C^2(\mathbb R^n)$ be a nonnegative entire solution of
\eqref{eq:main-equation} satisfying $|\nabla u|<1$ in $\mathbb R^n$.  If $u\not\equiv0$, then $u>0$ and $u\in C^\infty(\mathbb R^n)$.
\end{lemma}

\begin{proof}
Equation \eqref{eq:main-equation} can be written in nondivergence form as
\[
a^{ij}u_{ij}=-W^{-1}u^p\leqslant0,
\qquad
a^{ij}\coloneqq\delta_{ij}+W^2u_i u_j.
\]
On every bounded domain $\Omega\Subset\mathbb R^n$, strict spacelikeness and
continuity give $\sup_{\overline\Omega}|\nabla u|<1$, so
$a^{ij}\partial_{ij}$ is uniformly elliptic on $\Omega$. The strong maximum
principle \cite[Theorem~3.5]{GT} implies that a nontrivial nonnegative solution
is positive. It is therefore bounded away from zero on compact subsets. The
coefficient map $\xi\mapsto\xi/\sqrt{1-|\xi|^2}$ and the nonlinearity
$t\mapsto t^p$ are smooth on the respective local ranges of $\nabla u$ and
$u$. Standard interior Schauder bootstrapping for quasilinear elliptic
equations now yields $u\in C^\infty(\mathbb R^n)$.
\end{proof}

In the remainder of the paper, the trivial solution requires no further
consideration, and Lemma~\ref{lem:regularity-reduction} allows us to work with
a positive smooth solution.

Let $M\coloneqq\{(x,u(x))\mid x\in\mathbb R^n\}$ denote the graph of $u$ in
Lorentz--Minkowski space $\mathbb R^{n,1}$, and let $F_i\coloneqq(e_i,u_i)$ be the
coordinate tangent vector fields. With respect to the graph coordinates, the
induced metric and its inverse are
$\bar g_{ij}\coloneqq\delta_{ij}-u_i u_j$ and
$\bar g^{ij}=\delta_{ij}+W^2u_i u_j$, respectively. We denote the
Levi--Civita connection of $\bar g$ by $\bar\nabla$. For a scalar function
$v$, $\bar\nabla v$ denotes its gradient with respect to $\bar g$, and
$\bar\nabla_i v=v_i$; thus
$(\bar\nabla v)^i=\bar g^{ij}v_j$,
$\bar\nabla_i\bar\nabla_jv=v_{ij}-\bar\Gamma_{ij}^{k}v_k$, and
$\Delta_{\bar g}v=\bar g^{ij}\bar\nabla_i\bar\nabla_jv$. Unless stated
otherwise, all tensor norms and inner products in this section are taken with
respect to $\bar g$.

With respect to the future-directed unit normal $\nu\coloneqq W(\nabla u,1)$, define
\[
h_{ij}\coloneqq-\langle\mathrm D_{F_i}F_j,\nu\rangle_{\mathbb R^{n,1}}.
\]
The second fundamental form and its trace are
\[
h_{ij}=Wu_{ij},
\qquad
H\coloneqq\bar g^{ij}h_{ij}=\operatorname{div}(W\nabla u)=-u^p.
\]

We next record the identities used in the Bernstein argument. From
\[
\partial_i\bar g_{j\ell}+\partial_j\bar g_{i\ell}
-\partial_\ell\bar g_{ij}=-2u_{ij}u_\ell
\]
and $\bar g^{k\ell}u_\ell=W^2u_k$, we obtain
\begin{equation}\label{eq:christoffel}
\bar\Gamma_{ij}^{k}=-W^2u_k u_{ij}.
\end{equation}
Consequently,
\begin{align}
\begin{split}\label{eq:graph-hessian}
\bar\nabla_i\bar\nabla_j u
&=u_{ij}-\bar\Gamma_{ij}^{k}u_k=W^2u_{ij}=Wh_{ij},\\
(\bar\nabla u)^i&=\bar g^{ij}u_j=W^2u_i,\\
|\bar\nabla u|_{\bar g}^2
&=\bar g^{ij}u_i u_j=W^2|\nabla u|^2=W^2-1.
\end{split}
\end{align}
A direct computation gives $W_i=W^3u_k u_{ki}=h_{ij}\bar\nabla^ju$, and hence
\begin{equation}\label{eq:gradient-W}
\bar\nabla_iW=h_{ij}\bar\nabla^ju.
\end{equation}

Since the ambient Lorentz--Minkowski space is flat, the Codazzi identity gives
\begin{equation}\label{eq:codazzi}
\bar\nabla^ih_{ij}=\bar\nabla_jH.
\end{equation}
Using \eqref{eq:graph-hessian}--\eqref{eq:codazzi} and $H=-u^p$, we obtain
\begin{align}
\begin{split}\label{eq:laplacian-W}
\Delta_{\bar g}W
&=\bar\nabla^i\bigl(h_{ij}\bar\nabla^ju\bigr)
=(\bar\nabla^ih_{ij})\bar\nabla^ju
+h_{ij}\bar\nabla^i\bar\nabla^ju\\
&=\langle\bar\nabla H,\bar\nabla u\rangle_{\bar g}
+W|h|_{\bar g}^2
=W|h|_{\bar g}^2-pu^{p-1}(W^2-1).
\end{split}
\end{align}
Taking the trace of the first identity in \eqref{eq:graph-hessian} yields
\begin{equation}\label{eq:laplacian-u}
\Delta_{\bar g}u=WH=-Wu^p.
\end{equation}
From \eqref{eq:gradient-W} and \eqref{eq:laplacian-W}, we obtain
\begin{equation}\label{eq:laplacian-log-W}
\Delta_{\bar g}\log W
=|h|_{\bar g}^2
-\frac{|h(\bar\nabla u,\cdot)|_{\bar g}^2}{W^2}
-\frac{pu^{p-1}(W^2-1)}{W}.
\end{equation}

Let $x_k$ denote the restriction to $M$ of the $k$-th Euclidean coordinate
function. Formula \eqref{eq:christoffel} gives
\begin{equation}\label{eq:coordinate-identities}
\Delta_{\bar g}x_k=WHu_k,
\qquad
\sum_{k=1}^n|\bar\nabla x_k|_{\bar g}^2
=\sum_{k=1}^n\bar g^{kk}=n+W^2-1.
\end{equation}

\section{A global Bernstein estimate}\label{sec:Bernstein}

Throughout this section, we use the notation introduced in Section~\ref{sec:geometric}.

\begin{theorem}\label{thm:bernstein}
    Let $n\geqslant2$ and $p\geqslant1$, and let $u$ be a nonnegative, entire, strictly spacelike solution as in Section~\ref{sec:geometric}. Then there exists a constant $C=C(n,p)>0$ such that
    \[
W(x)\leqslant C(1+u(x))^{p+1},\quad x\in\mathbb R^n.
\]
\end{theorem}

\begin{proof}
    Fix $x_0\in\mathbb R^n$ and set $R\coloneqq1+u(x_0)$, $\tau_0\coloneqq u-u(x_0)$, and $z\coloneqq x-x_0$. Since $u\geqslant0$, we have $R\geqslant1$. For $z\neq0$, the strict spacelike condition gives
    \[
|\tau_0(x)|=\Big|\int_0^1\nabla u(x_0+tz)\cdot z\,\mathrm dt\Big|\leqslant|z|\int_0^1|\nabla u(x_0+tz)|\,\mathrm dt<|z|.
\]
    Thus, the Lorentzian squared distance $\rho^2\coloneqq |z|^2-\tau_0^2$ is nonnegative and is strictly positive away from $x_0$.

    We first record some elementary properties of the cutoff
    \begin{equation}\label{eq:theta-def}
        \Theta\coloneqq1-\frac{\rho^2}{R^2}-\frac{\tau_0}{R}.
    \end{equation}
    Since $\Theta(x_0)=1$, the set $\{\Theta>0\}$ is nonempty. If $\Theta(x)>0$, then $0\leqslant\rho^2<R^2-R\tau_0$. In particular, $R^2-R\tau_0>0$, and hence $\tau_0<R$. On the other hand, the nonnegativity of $u$ gives
    \[
\tau_0=u-u(x_0)\geqslant-u(x_0)=1-R>-R.
\]
    Consequently, $|z|^2=\rho^2+\tau_0^2<R^2-R\tau_0+\tau_0^2<3R^2$, since $\tau_0\in(-R,R)$ and the quadratic $s^2-Rs+R^2$ attains its maximum value $3R^2$ on $[-R,R]$. Furthermore, $u=R-1+\tau_0<2R-1<2R$ and $0<\Theta\leqslant1-\frac{\tau_0}{R}<2$ by $\rho^2\geqslant0$. Thus, in the region $\Omega\coloneqq \{\Theta>0\}$,
    \begin{equation}\label{eq:theta-region}
        \Omega\subset B_{\sqrt3R}(x_0),\quad -R<\tau_0<R,\quad 0\leqslant u<2R,\quad 0<\Theta<2.
    \end{equation}
    In particular, $\Omega$ is a bounded open set containing $x_0$, and continuity implies that $\Theta=0$ on $\partial\Omega$.

    Taking the $\bar g$-gradient of \eqref{eq:theta-def} yields
    \begin{equation}\label{eq:gradient-theta}
        \bar\nabla\Theta=-\frac{2}{R^2}\sum_{k=1}^n z_k\bar\nabla x_k+\big(\frac{2\tau_0}{R^2}-\frac 1 R\big)\bar\nabla u.
    \end{equation}

    Consider $G\coloneqq\Theta^\vartheta\frac{W}{(1+u)^{p+1}}$, where $\vartheta\coloneqq8(n-1)+16(p+1)$. Because $\overline\Omega$ is compact and $u$ is strictly spacelike, $W$ is bounded on $\overline\Omega$. Since $1+u\geqslant1$, the function $G$ extends continuously to $\overline\Omega$ by setting $G=0$ on $\partial\Omega$. Moreover, $G(x_0)=\frac{W(x_0)}{(1+u(x_0))^{p+1}}>0$. It follows that $G$ attains a positive maximum at an interior point $x_*\in\Omega$. All quantities below are evaluated at $x_*$. Since $\log G=\vartheta\log\Theta+\log W-(p+1)\log(1+u)$, the first-order condition is
    \begin{equation}\label{eq:maximum-first-order}
        \frac{\bar\nabla W}{W}=\frac{p+1}{1+u}\bar\nabla u-\frac{\vartheta}{\Theta}\bar\nabla\Theta.
    \end{equation}
    If $W(x_*)=1$, then $G(x_*)\leqslant 2^\vartheta$, and the desired estimate follows immediately. We may therefore assume that $W(x_*)>1$ and choose a $\bar g$-orthonormal frame such that $e_1\coloneqq\frac{\bar\nabla u}{\sqrt{W^2-1}}$. Define
    $\lambda\coloneqq-\frac{\vartheta}{\Theta}
    \frac{\langle\bar\nabla\Theta,\bar\nabla u\rangle_{\bar g}}{W^2-1}$. Using \eqref{eq:gradient-W}, the components of \eqref{eq:maximum-first-order} give
    \begin{equation}\label{eq:h11-at-maximum}
        h_{11}=W\big(\frac{p+1}{1+u}+\lambda\big),
    \end{equation}
    \begin{equation}\label{eq:maximum-components}
        -\frac{\vartheta}{\Theta}\bar\nabla_1\Theta=\sqrt{W^2-1}\,\lambda,\quad-\frac{\vartheta}{\Theta}\bar\nabla_\alpha\Theta=\frac{\sqrt{W^2-1}}{W}h_{1\alpha},\quad 2\leqslant\alpha\leqslant n.
    \end{equation}
    At the maximum point $x_*$, equation \eqref{eq:laplacian-log-W} and the definition of $\log G$ give
    \begin{align}
        \begin{split}\label{eq:maximum-laplacian}
            0\geqslant\Delta_{\bar g}\log G
            =\,&\frac{\vartheta}{\Theta}\Delta_{\bar g}\Theta
            -\frac{\vartheta}{\Theta^2}|\bar\nabla\Theta|_{\bar g}^2+|h|_{\bar g}^2
            -\frac{|h(\bar\nabla u,\cdot)|_{\bar g}^2}{W^2}\\
            &-\frac{pu^{p-1}(W^2-1)}{W}
            +\frac{(p+1)Wu^p}{1+u}
            +\frac{(p+1)(W^2-1)}{(1+u)^2}.
        \end{split}
    \end{align}
    The last two terms follow from
    \[
\Delta_{\bar g}\big(-(p+1)\log(1+u)\big)=-\frac{p+1}{1+u}\Delta_{\bar g}u+\frac{p+1}{(1+u)^2}|\bar\nabla u|_{\bar g}^2,
\]
    together with $\Delta_{\bar g}u=-Wu^p$ and $|\bar\nabla u|_{\bar g}^2=W^2-1$. We now estimate the terms in \eqref{eq:maximum-laplacian} separately.

    \paragraph{The cutoff-Laplacian term.}
    By \eqref{eq:laplacian-u} and \eqref{eq:coordinate-identities},
    $\Delta_{\bar g}\tau_0=W H$, and hence
    \[
\Delta_{\bar g}\rho^2=2n+2W H(z\cdot\nabla u-\tau_0),
\]
    \begin{equation}\label{eq:laplacian-theta}
        \Delta_{\bar g}\Theta=-\frac{2n}{R^2}-\frac{2W H}{R^2}(z\cdot\nabla u-\tau_0)-\frac{W H}{R}.
    \end{equation}
    To relate the last two terms to the identity at the maximum point, observe that
    \[
\langle\bar\nabla x_k,\bar\nabla u\rangle_{\bar g}=\bar g^{kj}u_j=W^2u_k,\quad|\bar\nabla u|_{\bar g}^2=W^2-1.
\]
    Using \eqref{eq:gradient-theta}, we therefore obtain
    \[
\langle\bar\nabla\Theta,\bar\nabla u\rangle_{\bar g}=-\frac{2}{R^2}\big(W^2z\cdot\nabla u-\tau_0(W^2-1)\big)-\frac{W^2-1}{R}.
\]
    The definition of $\lambda$ now gives
    \[
\frac{\Theta(W^2-1)}{\vartheta}\lambda=\frac{2}{R^2}\big(W^2z\cdot\nabla u-\tau_0(W^2-1)\big)+\frac{W^2-1}{R}.
\]
    Adding $R^{-1}-2\tau_0 R^{-2}$ to the right-hand side and dividing by $W^2$ yields the exact identity
    \[
\frac{2}{R^2}(z\cdot\nabla u-\tau_0)+\frac 1 R=\frac1{W^2}\big(\frac{\Theta(W^2-1)}{\vartheta}\lambda+\frac 1 R-\frac{2\tau_0}{R^2}\big).
\]
    Substituting this identity together with $H=-u^p$ into \eqref{eq:laplacian-theta}, we obtain
    \begin{equation}\label{eq:laplacian-theta-exact}
        \frac{\vartheta}{\Theta}\Delta_{\bar g}\Theta=\frac{(W^2-1)u^p}{W}\lambda-\frac{2n\vartheta}{R^2\Theta}+\frac{\vartheta u^p}{W\Theta}
        \big(\frac 1 R-\frac{2\tau_0}{R^2}\big).
    \end{equation}
    By \eqref{eq:theta-region}, $\big|\frac 1 R-\frac{2\tau_0}{R^2}\big|\leqslant\frac 3 R$. Since $R\geqslant1$, $W\geqslant1$, and $0<\Theta<2$, we have $\frac{1}{R^2\Theta}\leqslant\frac{2}{\Theta^2}$, $\frac{u^p}{RW\Theta}\leqslant\frac{C(1+u^{2p})}{\Theta^2}$. Consequently, \eqref{eq:laplacian-theta-exact} implies
    \begin{equation}\label{eq:laplacian-theta-lower}
        \frac{\vartheta}{\Theta}\Delta_{\bar g}\Theta\geqslant\frac{(W^2-1)u^p}{W}\lambda-\frac{C(1+u^{2p})}{\Theta^2}.
    \end{equation}

    \paragraph{The cutoff-gradient term.}
    The component identities \eqref{eq:maximum-components} give

    \begin{equation}\label{eq:gradient-theta-square}
        \frac{\vartheta}{\Theta^2}|\bar\nabla \Theta|_{\bar g}^2=\frac{W^2-1}{\vartheta}\lambda^2+\frac{W^2-1}{\vartheta W^2}\sum_{\alpha=2}^n h_{1\alpha}^2.
    \end{equation}

    \paragraph{Curvature terms.}
    For the third group in \eqref{eq:maximum-laplacian}, the trace identity
    $\sum_{\alpha=2}^n h_{\alpha\alpha}=-u^p-h_{11}$ and Cauchy's inequality
    imply
    \begin{equation}\label{eq:transverse-curvature}
        \sum_{\alpha,\gamma=2}^n h_{\alpha\gamma}^2\geqslant\frac{(u^p+h_{11})^2}{n-1}.
    \end{equation}
    Moreover, in the chosen orthonormal frame,
    \[
|h|_{\bar g}^2-\frac{|h(\bar\nabla u,\cdot)|_{\bar g}^2}{W^2}=\frac{h_{11}^2}{W^2}+\big(1+\frac1{W^2}\big)\sum_{\alpha=2}^n h_{1\alpha}^2+\sum_{\alpha,\gamma=2}^n h_{\alpha\gamma}^2.
\]
    Combining this identity with \eqref{eq:h11-at-maximum}, \eqref{eq:gradient-theta-square}, and \eqref{eq:transverse-curvature}, and retaining the principal $\lambda$-term from \eqref{eq:laplacian-theta-lower}, we obtain
    \begin{align}
        \begin{split}\label{eq:curvature-lower}
            &|h|_{\bar g}^2-\frac{|h(\bar\nabla u,\cdot)|_{\bar g}^2}{W^2}-\frac{\vartheta}{\Theta^2}|\bar\nabla \Theta|_{\bar g}^2+\frac{(W^2-1)u^p}{W}\lambda\\
            \geqslant\,&\big(\frac{p+1}{1+u}+\lambda\big)^2
            +\frac{1}{n-1}\big[u^p+W\big(\frac{p+1}{1+u}+\lambda\big)\big]^2\\
            &-\frac{W^2-1}{\vartheta}\lambda^2
            +\frac{(W^2-1)u^p}{W}\lambda.
        \end{split}
    \end{align}
    Indeed, the omitted coefficient of $\sum_{\alpha=2}^n h_{1\alpha}^2$ is $1+\frac1{W^2}-\frac{W^2-1}{\vartheta W^2}>0$ whenever $\vartheta>1$.

    Let $\mathcal Q$ denote the expression on the right-hand side of \eqref{eq:curvature-lower}, and let
    \[
\mathcal R\coloneqq\frac{(p+1)Wu^p}{1+u}-\frac{pu^{p-1}(W^2-1)}{W}+\frac{(p+1)(W^2-1)}{(1+u)^2}.
\]
    Substituting \eqref{eq:laplacian-theta-lower} into \eqref{eq:maximum-laplacian}, using the exact identity \eqref{eq:gradient-theta-square}, and then applying \eqref{eq:curvature-lower}, gives the intermediate estimate
    \begin{equation}\label{eq:QR-reduction}
        \Delta_{\bar g}\log G\geqslant\mathcal Q+\mathcal R-\frac{C(1+u^{2p})}{\Theta^2}.
    \end{equation}
    Thus it remains to estimate $\mathcal Q$ and $\mathcal R$ separately.

    \paragraph{The quadratic part $\mathcal Q$.}
    The elementary inequalities
    \[
    \begin{split}
    \left[u^p+W\left(\frac{p+1}{1+u}+\lambda\right)\right]^2
    &\geqslant\frac{W^2}{2}\left(\frac{p+1}{1+u}+\lambda\right)^2-u^{2p},\\
    \lambda^2&\leqslant2\left(\frac{p+1}{1+u}+\lambda\right)^2
    +\frac{2(p+1)^2}{(1+u)^2}
    \end{split}
    \]
    and $\vartheta\geqslant8(n-1)$ leave a coefficient at least
    $1+W^2/[4(n-1)]$ in front of
    $[(p+1)/(1+u)+\lambda]^2$ in $\mathcal Q$.
    Write the remaining linear term as
    \[
    u^p\frac{W^2-1}{W}\lambda
    =u^p\frac{W^2-1}{W}\left(\frac{p+1}{1+u}+\lambda\right)
    -\frac{(p+1)u^p(W^2-1)}{W(1+u)}.
    \]
    Completing the square and using $(W^2-1)/W\leqslant W$ gives
    \begin{equation}\label{eq:Q-first-lower}
    \mathcal Q\geqslant
    -\frac{2(p+1)^2}{\vartheta}\frac{W^2-1}{(1+u)^2}
    -\frac{(p+1)u^pW}{1+u}-Cu^{2p}.
    \end{equation}
    Since $\vartheta\geqslant16(p+1)$, the first coefficient is at most
    $(p+1)/8$. Also, $W\leqslant1+\sqrt{W^2-1}$ and Young's inequality give
    \[
    \frac{(p+1)u^pW}{1+u}
    \leqslant\frac{p+1}{8}\frac{W^2-1}{(1+u)^2}+C(1+u^{2p}).
    \]
    Therefore
    \begin{equation}\label{eq:Q-final-lower}
    \mathcal Q\geqslant
    -\frac{p+1}{4}\frac{W^2-1}{(1+u)^2}-C(1+u^{2p}).
    \end{equation}

    \paragraph{The reaction terms.}
    It remains to estimate
    \[
\mathcal R=\frac{(p+1)Wu^p}{1+u}-\frac{pu^{p-1}(W^2-1)}{W}+\frac{(p+1)(W^2-1)}{(1+u)^2}.
\]
    The first term is nonnegative, while the third is exactly $\frac{(p+1)(W^2-1)}{(1+u)^2}$. It remains to control the negative middle term. Since $\frac{W^2-1}{W}\leqslant\sqrt{W^2-1}$, we write
    \[
pu^{p-1}\sqrt{W^2-1}=p u^{p-1}(1+u)\frac{\sqrt{W^2-1}}{1+u}.
\]
    This is where the assumption $p\geqslant1$ is used: $u^{2p-2}(1+u)^2\leqslant C(1+u^{2p})$ for every $u\geqslant0$. Young's inequality therefore yields $\frac{pu^{p-1}(W^2-1)}{W}\leqslant\frac{p+1}{4}\frac{W^2-1}{(1+u)^2}+C(1+u^{2p})$. Substituting this estimate into $\mathcal R$ gives
    \begin{equation}\label{eq:R-lower}
        \mathcal R\geqslant\frac{3(p+1)}4\frac{W^2-1}{(1+u)^2}-C(1+u^{2p}).
    \end{equation}
    Combining \eqref{eq:QR-reduction}, \eqref{eq:Q-final-lower}, and \eqref{eq:R-lower}, we find
    \[
\Delta_{\bar g}\log G\geqslant\frac{p+1}{2}\frac{W^2-1}{(1+u)^2}-C(1+u^{2p})-\frac{C(1+u^{2p})}{\Theta^2}.
\]
    Since $0<\Theta<2$ and $\Delta_{\bar g}\log G(x_*)\leqslant0$, we obtain $0\geqslant\frac{p+1}{2}\frac{W^2-1}{(1+u)^2}-\frac{C(1+u^{2p})}{\Theta^2}$. Therefore, $\Theta^2\frac{W^2-1}{(1+u)^2}\leqslant C(1+u^{2p})$. Since $W\leqslant1+\sqrt{W^2-1}$ and $0<\Theta<2$, it follows that $\frac{\Theta W}{1+u}\leqslant C(1+u^p)$. At the maximum point,
    \[
G(x_*)=\Theta^{\vartheta-1}\frac{\Theta W}{1+u}\frac1{(1+u)^p}\leqslant C\Theta^{\vartheta-1}\frac{1+u^p}{(1+u)^p}\leqslant C.
\]
    Since $x_*$ is a maximizer and $\Theta(x_0)=1$, we have
    \[
\frac{W(x_0)}{(1+u(x_0))^{p+1}}=G(x_0)\leqslant G(x_*)\leqslant C.
\]
    Because $x_0\in\mathbb R^n$ was arbitrary, the theorem follows.
\end{proof}

\section{A uniform height bound}\label{sec:bound}

\begin{lemma}\label{lem:height-bound}
    Let $n\geqslant2$ and $p\geqslant1$, and let $u$ be a nonnegative, entire, strictly spacelike solution as in Section~\ref{sec:geometric}. Then there exists a constant $C=C(n,p)>0$ such that $\sup_{\mathbb R^n}u\leqslant C$.
\end{lemma}

\begin{proof}
    If $u\leqslant1$ in $\mathbb R^n$, the conclusion is immediate. Fix a point $y\in\mathbb R^n$ with $u(y)>1$, and set $R\coloneqq u(y)-1>0$. The strict spacelike condition implies $u(x)\geqslant u(y)-|x-y|$, and hence
    \begin{equation}\label{eq:height-unit-ball}
        u\geqslant1\quad\text{in}\quad\overline{B_R(y)}.
    \end{equation}

    Set $m\coloneqq \min_{\partial B_R(y)}u\geqslant1$. The minimum of $u$ on $\overline{B_R(y)}$ cannot be attained in the interior. Indeed, at an interior minimum we would have $\nabla u=0$ and $\nabla^2 u\geqslant0$, and hence $\operatorname{div}\big(\frac{\nabla u}{\sqrt{1-|\nabla u|^2}}\big)=\Delta u\geqslant0$. On the other hand, the equation would reduce at that point to $\Delta u=-u^p<0$, because $u\geqslant1$ in $\overline{B_R(y)}$ by \eqref{eq:height-unit-ball}. This is a contradiction.
    It follows that
    \begin{equation}\label{eq:boundary-minimum}
        u\geqslant m\quad\text{in }\overline{B_R(y)},\quad u>m\quad\text{in }B_R(y).
    \end{equation}

    We compare $u$ with a radial hyperbolic cap. For $r\coloneqq|x-y|$, define
    \begin{equation}\label{eq:hyperbolic-cap}
        b(r)\coloneqq m+\int_r^R\frac{m^p s/n}{\sqrt{1+(m^p s/n)^2}}\,\mathrm ds.
    \end{equation}
    Then $b'(r)=-\frac{m^p r/n}{\sqrt{1+(m^p r/n)^2}}$ and $|b'(r)|<1$. Thus, $b$ is a smooth spacelike radial function on $\overline{B_R(y)}$ satisfying $b=m$ on $\partial B_R(y)$. Moreover, $\frac{b'(r)}{\sqrt{1-b'(r)^2}}=-\frac{m^p}{n}r$, and hence
    \begin{equation}\label{eq:cap-equation}
        \operatorname{div}\big(\frac{\nabla b}{\sqrt{1-|\nabla b|^2}}\big)=\operatorname{div}\big(-\frac{m^p}{n}(x-y)\big)=-m^p.
    \end{equation}

    For $\xi\in B_1(0)$, define $A(\xi)\coloneqq\frac{\xi}{\sqrt{1-|\xi|^2}}$ and $\mathcal M(v)\coloneqq\operatorname{div}(A(\nabla v))$. Set $w\coloneqq u-b$. Since $p\geqslant1$ and $u\geqslant m$ by \eqref{eq:boundary-minimum}, equations \eqref{eq:cap-equation} and \eqref{eq:main-equation} give
    \begin{equation}\label{eq:operator-comparison}
        \mathcal M(u)-\mathcal M(b)=-u^p+m^p\leqslant0.
    \end{equation}
    The mean value formula gives
    \begin{equation}\label{eq:linearization-matrix}
        A(\nabla u)-A(\nabla b)=B(x)\nabla w,\quad\text{where}\quad B(x)\coloneqq\int_0^1\mathrm DA\big(\nabla b+t(\nabla u-\nabla b)\big)\,\mathrm dt.
    \end{equation}
    For $|\xi|<1$, $\mathrm DA(\xi)=\frac1{\sqrt{1-|\xi|^2}}I+\frac{\xi\otimes\xi}{(1-|\xi|^2)^{3/2}}$. In particular, $\eta^T\mathrm DA(\xi)\eta\geqslant|\eta|^2$ for every $\eta\in\mathbb R^n$. The unit ball is convex, so every argument of $\mathrm DA$ in \eqref{eq:linearization-matrix} lies in $B_1(0)$. Moreover, compactness of $\overline{B_R(y)}$ and strict spacelikeness imply that these arguments remain in a compact subset of $B_1(0)$. Consequently, $B(x)$ is bounded and uniformly positive definite.

    Taking the divergence of \eqref{eq:linearization-matrix} and using \eqref{eq:operator-comparison}, we obtain $\operatorname{div}(B(x)\nabla w)\leqslant0$ in $B_R(y)$. On the boundary, $w=u-m\geqslant0$. The weak minimum principle therefore yields
    \begin{equation}\label{eq:cap-comparison}
        u\geqslant b\quad\text{in}\quad B_R(y).
    \end{equation}
    Evaluating \eqref{eq:cap-comparison} at the center and using \eqref{eq:hyperbolic-cap}, we obtain
    \[
u(y)\geqslant b(0)=m+\frac{n}{m^p}\big(\sqrt{1+(\frac{m^pR}{n})^2}-1\big)\geqslant m+R-\frac{n}{m^p}.
\]
    Since $u(y)=R+1$, it follows that $m\leqslant1+\frac{n}{m^p}$. Together with $m\geqslant1$, this gives
    \begin{equation}\label{eq:boundary-height-bound}
        1\leqslant m\leqslant n+1.
    \end{equation}

    Choose $\hat x\in\partial B_R(y)$ such that $u(\hat x)=m$, and let $\nu_{\mathrm E}$ denote the outward Euclidean unit normal to $\partial B_R(y)$ at $\hat x$. By \eqref{eq:cap-comparison}, $w\geqslant0$ in $B_R(y)$ and $w(\hat x)=0$. For all sufficiently small $s>0$, we have $w(\hat x-s\nu_{\mathrm E})\geqslant w(\hat x)=0$. Thus, the one-sided derivative at $s=0$ satisfies $-\partial_{\nu_{\mathrm E}} w(\hat x)\geqslant0$, or equivalently $\partial_{\nu_{\mathrm E}} w(\hat x)\leqslant0$. Hence $u_{\nu_{\mathrm E}}(\hat x)\leqslant b'(R)=-\frac{m^pR/n}{\sqrt{1+(m^pR/n)^2}}<0$. Since $|\nabla u(\hat x)|\geqslant|u_{\nu_{\mathrm E}}(\hat x)|$, it follows that
    \begin{equation}\label{eq:W-lower-bound}
        W(\hat x)=\frac1{\sqrt{1-|\nabla u(\hat x)|^2}}\geqslant\frac1{\sqrt{1-u_{\nu_{\mathrm E}}(\hat x)^2}}\geqslant\sqrt{1+(\frac{m^pR}{n})^2}.
    \end{equation}
    On the other hand, the Bernstein estimate of Section~\ref{sec:Bernstein} and \eqref{eq:boundary-height-bound} give
    \begin{equation}\label{eq:W-upper-bound}
        W(\hat x)\leqslant C(n,p)(1+m)^{p+1}\leqslant C(n,p).
    \end{equation}
    Since $m\geqslant1$, inequalities \eqref{eq:W-lower-bound} and \eqref{eq:W-upper-bound} imply $\sqrt{1+(R/n)^2}\leqslant C(n,p)$, and hence $R\leqslant C(n,p)$. Therefore, $u(y)=R+1\leqslant C(n,p)$, which bounds $u(y)$ in terms of $n$ and $p$ whenever $u(y)>1$. The remaining points already satisfy $u(y)\leqslant1$, and the lemma follows.
\end{proof}

\begin{proof}[Proof of Theorem~\ref{thm:universal-bounds}]
Lemma~\ref{lem:height-bound} and Theorem~\ref{thm:bernstein} give
$0\leqslant u\leqslant C(n,p)$ and $1\leqslant W\leqslant C(n,p)$.
Since $W=(1-|\nabla u|^2)^{-1/2}$, it follows that
\[
|\nabla u|\leqslant\sqrt{1-C^{-2}}<1
\qquad\text{in }\mathbb R^n.
\]
This proves all the assertions of Theorem~\ref{thm:universal-bounds}.
\end{proof}

\section{Notation and the differential identity}
\label{sec:identities}

Let $n\geqslant2$, $p\geqslant1$, and let $u>0$ be a smooth entire
strictly spacelike solution of \eqref{eq:main-equation}. We record only
the notation and pointwise identities needed for the common starting
point, Lemma~\ref{lem:liouville-defect}, of the Liouville proof in
Section~\ref{sec:liouville}. Set
\[
H\coloneqq\operatorname{div}(W\nabla u)=-u^p,\qquad X_i\coloneqq Wu_i,
\qquad\beta\coloneqq-\frac{2p}{n+2},
\]
and define
\begin{equation}\label{eq:tensor-definitions}
K_{ij}\coloneqq X_{i,j}-\frac Hn\delta_{ij},\qquad
L_{ij}\coloneqq\frac{X_iu_j}{u}-\frac{|\nabla u|^2}{nu}W\delta_{ij}.
\end{equation}
All contractions are Euclidean, and repeated indices are summed.
The tensor $L$ is symmetric, whereas $K$ need not be.

\begin{lemma}\label{lem:f-properties}
For $n\geqslant2$ and $p>0$, the function
\begin{equation}\label{eq:f-definition}
f(s)\coloneqq-\frac{np}{n+2}
\frac{\Gamma(\frac n2+2)}{\Gamma(\frac{n+3}{2})\sqrt\pi}
\int_0^1\frac{t^{\frac{n+1}{2}}(1-t)^{-\frac12}}{1-s+st}
\,\mathrm dt,\qquad 0\leqslant s\leqslant1,
\end{equation}
belongs to $C^1([0,1])$, is strictly decreasing with $f'(s)<0$, and satisfies
\[
f(0)=-\frac{np}{n+2},\qquad f(1)=-\frac{np}{n+1},
\]
\begin{equation}\label{eq:f-ode}
np+(n+2-s)f(s)+2s(1-s)f'(s)=0.
\end{equation}
\end{lemma}

\begin{proof}
Since $1-x+xt\geqslant t$, the
absolute value of the differentiated integrand is bounded by
$t^{(n-3)/2}(1-t)^{1/2}\in L^1(0,1)$, giving the asserted regularity
and sign. The endpoint values follow from the beta-function identity.
Integrating the derivative of
$2t^{(n+3)/2}(1-t)^{1/2}/(1-x+xt)$ over $(0,1)$ and using the same
beta-function identity gives \eqref{eq:f-ode}.
\end{proof}

Define
\begin{equation}\label{eq:h-definitions}
h_1(s)\coloneqq\frac12\bigl(\beta-(n-1)[p+f(s)]\bigr),\qquad
h(s)\coloneqq4\bigl[(\beta-1)f(s)-h_1(s)^2\bigr],
\end{equation}
and set 
\begin{align}
	E_{ij}\coloneqq K_{ij}+h_1(|\nabla u|^2)L_{ij}.
	\end{align}
We always write $h_1 = h_1(|\nabla u|^2)$, $f = f(|\nabla u|^2)$ and $h = h(|\nabla u|^2)$ below whenever they are
unambiguous.

\begin{proposition}\label{prop:tensor-identity}
The following identity and pointwise estimates hold:
\begin{equation}\label{eq:tensor-identity}
u^{-\beta}\partial_i\bigl\{u^\beta(K_{ij}+fL_{ij})X_j\bigr\}
=E_{ij}E_{ji}+\frac h4L_{ij}L_{ij},
\end{equation}
\begin{equation}\label{eq:pointwise-tensor-bounds}
E_{ij}E_{ji}\geqslant0,\qquad
|EX|^2\leqslant W^2|\nabla u|^2E_{ij}E_{ji},\qquad
L_{ij}L_{ij}=\frac{n-1}{n}W^2u^{-2}|\nabla u|^4.
\end{equation}
\end{proposition}

\begin{proof}
Using $\partial_jH=pHu_j/u$,
$\partial_iK_{ij}=(n-1)\partial_jH/n$ and
$X_{i,j}=Wu_{ij}+W^3u_{kj}u_ku_i$, expand the divergence on the
left of \eqref{eq:tensor-identity}. The coefficient of
$H|\nabla u|^2W/u$ vanishes by \eqref{eq:f-ode}; the remaining terms are
\[
K_{ij}K_{ji}+2h_1K_{ij}L_{ij}+(\beta-1)fL_{ij}L_{ij}.
\]
Completing the square gives \eqref{eq:tensor-identity}.

At a point, choose orthonormal coordinates with $u_1=|\nabla u|$ and
diagonalize $(u_{ij})$ on the orthogonal complement. For
$\iota,\xi\in\{2,\ldots,n\}$, this gives
$E_{\iota\xi}=0$ if $\iota\ne\xi$ and
$E_{\iota1}=W^{-2}E_{1\iota}$. Consequently,
\[
E_{ij}E_{ji}\geqslant\sum_i|E_{i1}|^2,\qquad
|EX|^2=W^2|\nabla u|^2\sum_i|E_{i1}|^2.
\]
The same choice is arbitrary when $\nabla u=0$, and the conclusions
remain valid. The formula for $L_{ij}L_{ij}$ follows directly from
\eqref{eq:tensor-definitions}.
\end{proof}

\section{Proof of the Liouville theorem}
\label{sec:liouville}

We prove Theorem~\ref{thm:main} in the subcritical and critical ranges
using a core-counting estimate, and an integral-feedback
argument.

Assume that $u$ is a nontrivial nonnegative solution in either range. The regularity
reduction in Lemma~\ref{lem:regularity-reduction}, together with
Theorem~\ref{thm:universal-bounds}, gives
\begin{equation}\label{eq:liouville-bounds}
u\in C^\infty(\mathbb R^n),\qquad
0<u\leqslant C(n,p),\qquad 1\leqslant W\leqslant C(n,p).
\end{equation}
Constants denoted by $c,C>0$ may change from line to line. Unless additional
dependence is stated, they depend only on $n, p$.

\subsection{The linear-reaction and two-dimensional cases}

We first prove Theorem~\ref{thm:main} when $p=1$ or $n=2$.
Let $R>1$, and choose $\eta\in C_c^\infty(B_{2R})$ such that
$\eta\equiv1$ on $B_R$, $0\leqslant\eta\leqslant1$, and
$|\nabla\eta|\leqslant CR^{-1}$. Testing \eqref{eq:main-equation}
with $u^{-1}\eta^2$ and applying Cauchy--Schwarz gives
\begin{align}
\begin{split}\label{eq:logarithmic-test}
&\int u^{p-1}\eta^2+\int W\eta^2|\nabla\log u|^2\\
&\quad=2\int_{B_{2R}\setminus B_R}
W\eta\nabla\log u\cdot\nabla\eta\\
&\quad\leqslant
2\left(\int_{B_{2R}\setminus B_R}W\eta^2|\nabla\log u|^2\right)^{1/2}
\left(\int_{B_{2R}\setminus B_R}W|\nabla\eta|^2\right)^{1/2}.
\end{split}
\end{align}
Young's inequality and \eqref{eq:liouville-bounds} imply
\begin{equation}\label{eq:logarithmic-bound}
\int u^{p-1}\eta^2+\frac12\int W\eta^2|\nabla\log u|^2
\leqslant CR^{n-2}.
\end{equation}
If $p=1$, this gives $|B_R|\leqslant CR^{n-2}$, a contradiction
as $R\to\infty$.

If $n=2$ and $p>1$, \eqref{eq:logarithmic-bound} implies
$W|\nabla\log u|^2\in L^1(\mathbb R^2)$, so its integral over
$B_{2R}\setminus B_R$ tends to zero. The second factor in
\eqref{eq:logarithmic-test} is bounded independently of $R$.
Letting $R\to\infty$ and applying Fatou's lemma gives
$\int_{\mathbb R^2}u^{p-1}=0$, contrary to $u>0$.

\subsection{A common current and cutoff estimate}

The remaining range is
\begin{equation}\label{eq:liouville-range}
n\geqslant3,\qquad 1<p\leqslant\frac{n+2}{n-2},\qquad
\beta=-\frac{2p}{n+2}.
\end{equation}
Retain the notation of Section~\ref{sec:identities}, and define
\begin{equation}\label{eq:liouville-current}
\mathcal V_i\coloneqq u^\beta(K_{ij}+fL_{ij})X_j,
\qquad \mathcal D\coloneqq\operatorname{div}\mathcal V.
\end{equation}
Here $f,h_1,h$ are evaluated at $|\nabla u|^2$ when the argument is omitted.

\begin{lemma}\label{lem:liouville-defect}
Throughout \eqref{eq:liouville-range}, it holds that
\begin{equation}\label{eq:crit-defect-density}
\mathcal D=u^\beta\left(E_{ij}E_{ji}+\frac h4L_{ij}L_{ij}\right),
\end{equation}
and
\begin{align}
\mathcal D&\geqslant
c u^{\beta-2}|\nabla u|^6,
\label{eq:liouville-defect-estimates}\\*
|\mathcal V|^2&\leqslant C\mathcal D\,u^\beta|\nabla u|^2.
\label{eq:liouville-current-bound}
\end{align}
In particular, $\mathcal D\geqslant0$.
\end{lemma}

\begin{proof}
The identity is Proposition~\ref{prop:tensor-identity}. By
\eqref{eq:h-definitions} and the endpoint values in
Lemma~\ref{lem:f-properties},
\begin{align}
&h(0)=\frac{4np}{(n+2)^2}\bigl[(n+2)-(n-2)p\bigr]\geqslant0,
\\
&h(s)=h(0)+[f(0)-f(s)] \cdot\left\{4+\frac{4p(n^2-n+2)}{n+2}
-(n-1)^2[f(0)-f(s)]\right\},
\label{eq:crit-h-factor}\\
&f(s)-h_1(s)=-\frac{n+1}{2}[f(0)-f(s)].\label{eq:crit-h-factor1}
\end{align}
Set
\[
\delta(s):=f(0)-f(s).
\]
The integral formula \eqref{eq:f-definition} gives
\[
\frac{np}{(n+2)(n+4)}\,s
\leqslant \delta(s)
\leqslant
\frac{np}{(n+1)(n+2)}\,s,
\qquad 0\leqslant s\leqslant1.
\]
These inequalities follow by bounding the denominator \(1-s+st\) between
\(t\) and \(1\) and applying Lemma~\ref{lem:f-properties}. By
\eqref{eq:crit-h-factor},
\[
h(s)
=
h(0)+\delta(s)
\left[
4+\frac{4p(n^2-n+2)}{n+2}
-(n-1)^2\delta(s)
\right].
\]
Since \(s\leqslant1\), the upper bound for \(\delta(s)\) yields
\[
\begin{aligned}
4+\frac{4p(n^2-n+2)}{n+2}
-(n-1)^2\delta(s)
&\geqslant
4+\frac{p}{n+2}
\left[
4(n^2-n+2)
-\frac{n(n-1)^2}{n+1}
\right]
\\
&\geqslant 4.
\end{aligned}
\]
Hence
\[
h(s)\geqslant h(0)+4\delta(s)
\geqslant
h(0)+\frac{4np}{(n+2)(n+4)}\,s,
\]
which proves \eqref{eq:liouville-defect-estimates}.
Moreover, \eqref{eq:crit-h-factor1} gives
\[
f(s)-h_1(s)
=
-\frac{n+1}{2}\delta(s).
\]
Since \(\delta(s)\) is uniformly bounded on \([0,1]\) and
\(h(s)\geqslant4\delta(s)\), we obtain
\[
|f(s)-h_1(s)|^2
=
\frac{(n+1)^2}{4}\delta(s)^2
\leqslant C\delta(s)
\leqslant Ch(s).
\]
Applying \eqref{eq:pointwise-tensor-bounds} to
$\mathcal V=u^\beta\{EX+(f-h_1)LX\}$,
this proves \eqref{eq:liouville-current-bound}.
\end{proof}

For balls centered at the origin, write
\begin{equation}\label{eq:liouville-integrals}
\mathfrak D(R)\coloneqq\int_{B_R}\mathcal D\,\mathrm dx,\qquad
\mathcal E(R)\coloneqq\int_{B_R}u^\beta|\nabla u|^2\,\mathrm dx.
\end{equation}
Use the cutoff $\eta$ chosen above. Integration by parts and
\eqref{eq:liouville-current-bound} give
\[
\int\eta^2\mathcal D
=-2\int\eta\mathcal V\cdot\nabla\eta
\leqslant C
\left(\int_{B_{2R}\setminus B_R}\eta^2\mathcal D\right)^{1/2}
\left(R^{-2}\mathcal E(2R)\right)^{1/2}.
\]
Young's inequality and $\eta=1$ on $B_R$ therefore imply
\begin{align}
\mathfrak D(R)&\leqslant CR^{-2}\mathcal E(2R),
\label{eq:liouville-cutoff}\\*
\mathfrak D(R)&\leqslant
C[\mathfrak D(2R)-\mathfrak D(R)]^{1/2}
[R^{-2}\mathcal E(2R)]^{1/2}.
\label{eq:liouville-annular-cutoff}
\end{align}
If $\mathcal E(R)=O(R^2)$, the first inequality gives
$\mathfrak D(\infty)<\infty$, and the second then gives
$\mathfrak D\equiv0$ by letting $R\to\infty$.
By \eqref{eq:liouville-defect-estimates}, $\mathcal D=0$ forces
$\nabla u=0$, and the equation gives $u=0$. Thus
\begin{equation}\label{eq:liouville-criterion}
\mathcal E(R)=O(R^2)\quad\Longrightarrow\quad
\mathfrak D\equiv0\quad\Longrightarrow\quad u\equiv0.
\end{equation}

\subsection{Source bounds and Green kernels}

Write
\begin{equation}\label{eq:liouville-U}
\mathcal U(R)\coloneqq\int_{B_R}u^{p+1+\beta}\,\mathrm dx.
\end{equation}
Regard $W$ as a fixed scalar coefficient, and let $G_u(x,y)$ be the
positive whole-space Green kernel of $-\operatorname{div}(W\nabla\cdot)$.
Uniformly elliptic Green-kernel estimates
\cite[Theorem~7.1 and (7.9), pp.~66--67]{LSW1963} give
\begin{equation}\label{eq:crit-green-bounds}
c|x-y|^{2-n}\leqslant G_u(x,y)\leqslant C|x-y|^{2-n}.
\end{equation}
The corresponding Dirichlet kernel on $B_{2R}(z)$ is bounded below by
$cR^{2-n}$ when both variables lie in $B_R(z)$. By
\eqref{eq:liouville-bounds}, these constants are uniform in the center, the
radius, and the solution.

\begin{lemma}\label{lem:crit-source}
For every $z\in\mathbb R^n$, $R>0$ and $0<t\leqslant p$,
\begin{equation}\label{eq:crit-source-bounds}
\int_{B_R(z)}u^p\leqslant CR^{n-2p/(p-1)},\qquad
\int_{B_R(z)}u^t\leqslant C(n,p,t)R^{n-2t/(p-1)}.
\end{equation}
Moreover,
\begin{equation}\label{eq:crit-representation}
u(x)=\int_{\mathbb R^n}G_u(x,y)u(y)^p\,\mathrm dy.
\end{equation}
\end{lemma}

\begin{proof}
Set $m_R\coloneqq\int_{B_R(z)}u^p$, and let $G_{B_{2R}(z)}$ denote
the Dirichlet Green kernel of $-\operatorname{div}(W\nabla\cdot)$ on
$B_{2R}(z)$. Define the auxiliary potential
\[
w_R(x)\coloneqq\int_{B_R(z)}G_{B_{2R}(z)}(x,y)u(y)^p\,\mathrm dy.
\]
It is the weak solution of
\[
\begin{cases}
-\operatorname{div}(W\nabla w_R)=u^p\mathbf1_{B_R(z)}
&\text{in }B_{2R}(z),\\
w_R=0&\text{on }\partial B_{2R}(z).
\end{cases}
\]
The zero boundary condition is imposed only on $w_R$; no boundary
condition is prescribed for $u$. Indeed,
\[
-\operatorname{div}\bigl(W\nabla(u-w_R)\bigr)
=u^p\mathbf1_{B_{2R}(z)\setminus B_R(z)}\geqslant0
\quad\text{in }B_{2R}(z),
\]
and $u-w_R=u>0$ on $\partial B_{2R}(z)$. The weak minimum principle
therefore gives $u\geqslant w_R$. The Dirichlet Green-kernel lower bound
stated above yields
\[
u(x)\geqslant w_R(x)\geqslant cR^{2-n}m_R,
\qquad x\in B_R(z).
\]
Raising this inequality to the power $p$ and integrating over $B_R(z)$,
we obtain $m_R\geqslant cR^{n-p(n-2)}m_R^p$.
Since $m_R>0$ and $p>1$, division proves the first estimate;
the second follows by H\"older.

The increasing-ball Green potentials of the full source are bounded above
by $u$. Their limit leaves the nonnegative bounded entire $W$-harmonic
remainder
\[
\mathcal H(x)\coloneqq u(x)
-\int_{\mathbb R^n}G_u(x,y)u(y)^p\,\mathrm dy.
\]
Thus $0\leqslant\mathcal H\leqslant u$ and
$-\operatorname{div}(W\nabla\mathcal H)=0$ in $\mathbb R^n$.
Since $1\leqslant W\leqslant C$, the interior oscillation estimate gives,
for $0<r<R/2$,
\[
\operatorname{osc}_{B_r}\mathcal H
\leqslant C\left(\frac rR\right)^\alpha
\operatorname{osc}_{B_R}\mathcal H
\leqslant C\left(\frac rR\right)^\alpha\|u\|_{L^\infty(\mathbb R^n)},
\]
where $C>0$ and $\alpha\in(0,1)$ are independent of $r,R$.
Letting $R\to\infty$ with $r$ fixed shows that
$\operatorname{osc}_{B_r}\mathcal H=0$. As $r$ is arbitrary,
$\mathcal H\equiv\mathcal H(0)\geqslant0$.
If $\mathcal H(0)>0$, then $u\geqslant\mathcal H(0)$, and
\eqref{eq:crit-source-bounds} implies
\[
\mathcal H(0)^p|B_1|
\leqslant R^{-n}\int_{B_R}u^p
\leqslant CR^{-2p/(p-1)}\longrightarrow0
\qquad\text{as }R\to\infty,
\]
a contradiction. Hence $\mathcal H\equiv0$, proving
\eqref{eq:crit-representation}.
\end{proof}

Write
\[
\Phi(x)\coloneqq\frac{|x|^{2-n}}{(n-2)|\mathbb S^{n-1}|},\qquad -\Delta\Phi=\delta_0.
\]

\begin{lemma}\label{lem:crit-green-limit}
Let $u_j$ be positive entire solutions satisfying \eqref{eq:liouville-bounds},
let $z_j\in\mathbb R^n$, and let $t_j\to\infty$. Then
\begin{equation}\label{eq:crit-green-limit}
t_j^{n-2}G_{u_j}(z_j+t_jx,z_j+t_jy)\longrightarrow\Phi(x-y)
\end{equation}
locally uniformly off the set $\{ (x, y): x = y \}$.
\end{lemma}

\begin{proof}
The height bound and \eqref{eq:logarithmic-bound}, with translated
cutoffs, give
\begin{align*}
\int_{B_R(z)}(W-1)
&=\int_{B_R(z)}\frac{W^2-1}{W+1}\\
&\leqslant C\int_{B_R(z)}|\nabla u|^2
\leqslant C\int_{B_R(z)}|\nabla\log u|^2
\leqslant CR^{n-2}.
\end{align*}
Consequently, together with $W\leqslant C$, this gives
\begin{equation}\label{eq:crit-coefficients}
W[u_j](z_j+t_j\,\cdot)\longrightarrow1
\quad\text{in }L^q_{\mathrm{loc}}(\mathbb R^n),\quad 1\leqslant q<\infty,
\end{equation}
with $L^1$ error $O(t_j^{-2})$ and the other exponents following from
uniform boundedness.

Temporarily denote the left-hand side of \eqref{eq:crit-green-limit}
by $\widetilde G_j(x, y)$:
\begin{align}
-\operatorname{div}_x
\left(
W_j(x)\nabla_x\widetilde G_j(x,y)
\right)
=
\delta_y,
\end{align}
and
\begin{equation}\label{eq:crit-green-bounds111}
c|x-y|^{2-n}\leqslant \widetilde G_j(x,y)\leqslant C|x-y|^{2-n}.
\end{equation}
Here $W_j(x)\coloneqq W[u_j](z_j+t_jx)$. For fixed $y$, interior H\"older
estimates, symmetry, and \eqref{eq:crit-green-bounds} give locally uniform
subsequential convergence on every compact set
$K\Subset\mathbb R^n\setminus\{y\}$. Thus the limit function
$G_{\infty}(x,y)$ is harmonic for $x\neq y$. It remains to prove that
\begin{align}
-\Delta_xG_\infty(x,y)=\delta_y.
\end{align}
By \eqref{eq:crit-green-bounds111}, Caccioppoli's inequality on annuli
centered at $y$ gives
\[
\int_{B_{2r}(y)\setminus B_r(y)}
|\nabla_x\widetilde G_j(x,y)|^2\leqslant Cr^{2-n},\qquad
\int_{B_\rho(y)}|\nabla_x\widetilde G_j(x,y)|\leqslant C\rho.
\]
The second bound follows by dyadic summation. Away from $B_\rho(y)$,
weak $L^2$ compactness and \eqref{eq:crit-coefficients} allow passage
to the Green equation; the omitted part is
$O(\rho\|\nabla\varphi\|_\infty)$. Thus the limit retains source
$\delta_y$. Its difference from $\Phi(\cdot-y)$ is entire harmonic
and vanishes at infinity by \eqref{eq:crit-green-bounds}, hence is zero.
\end{proof}

\subsection{Natural scales and the initial growth bound}

For $n=3,4$, one has $\beta\geqslant-2$. Thus the height bound and
\eqref{eq:logarithmic-bound} give
\[
\mathcal E(R)=\int_{B_R}u^\beta|\nabla u|^2\,\mathrm dx \leqslant C\int_{B_R}|\nabla\log u|^2
\leqslant CR^{n-2}.
\]
The common criterion \eqref{eq:liouville-criterion} completes these cases.

For $n\geqslant5$, use the natural length
\begin{equation}\label{eq:crit-natural-length}
\ell(x)\coloneqq u(x)^{-(p-1)/2}\geqslant c>0.
\end{equation}
At the critical exponent $p=(n+2)/(n-2)$, this is $\ell=u^\beta$.
For a fixed center $y$, set $\ell\coloneqq\ell(y)$ and
\begin{align}\label{09201951}
v(z)\coloneqq\ell^{2/(p-1)}u(y+\ell z)
\end{align}
Then $v(0)=1$ and
\begin{equation}\label{eq:common-scaled-equation}
-\operatorname{div}
\frac{\nabla v}{\sqrt{1-\ell^{-2-4/(p-1)}|\nabla v|^2}}=v^p,
\end{equation}
with $1\leqslant(1-\ell^{-2-4/(p-1)}|\nabla v|^2)^{-1/2}\leqslant C$.

\begin{lemma}[Scale-invariant Harnack estimate]\label{lem:crit-harnack}
Let $v>0$ be a solution of the scaled equation \eqref{eq:common-scaled-equation} in a ball $B_{2r}(z_0)$. If
\begin{align}
r^2\sup_{B_{2r}(z_0)}v^{p-1}\leqslant M,
\label{eq:crit-harnack-potential}
\end{align}
then
\begin{align}
\sup_{B_r(z_0)}v
\leqslant
C\,\inf_{B_r(z_0)}v,
\label{eq:crit-harnack}
\end{align}
where $C=C(n,p,M)>0$ is independent of the center, the radius, the solution,
and the scaling parameter.

In particular, for every $K>1$ there exists $c=c(n,p,K)>0$ such that, if $v(0)=1$ and $v\leqslant K$ on $B_1$, then
\begin{align}
v\geqslant c
\quad\hbox{on }B_{1/2},
\qquad
\int_{B_{1/2}}v^p\,\mathrm dz\geqslant c.
\label{eq:crit-harnack-mass}
\end{align}
\end{lemma}

\begin{proof}
Write the scaled equation in the form
\begin{align}
-\operatorname{div}(a\nabla v)=v^{p-1}v,
\qquad
a(z)\coloneqq
\frac{1}{
\sqrt{1-\ell^{-2-4/(p-1)}|\nabla v(z)|^2}
}.
\label{eq:crit-harnack-linear}
\end{align}
By Theorem~\ref{thm:universal-bounds},
\begin{align}
1\leqslant a\leqslant C(n,p).
\label{eq:crit-harnack-ellipticity}
\end{align}
Thus \eqref{eq:crit-harnack-linear} is a uniformly elliptic divergence-form
equation with zeroth-order coefficient $v^{p-1}$.

Rescale $B_{2r}(z_0)$ to $B_2$ by setting
$w(\xi)\coloneqq v(z_0+r\xi)$. Then
\begin{align}
-\operatorname{div}\bigl(a_r\nabla w\bigr)
=
q_r w,
\qquad
a_r(\xi)\coloneqq a(z_0+r\xi),
\qquad
q_r(\xi)\coloneqq r^2v(z_0+r\xi)^{p-1}.
\end{align}
The coefficient $a_r$ has the same uniform ellipticity bounds as $a$, while
\eqref{eq:crit-harnack-potential} gives
\begin{align}
\|q_r\|_{L^\infty(B_2)}\leqslant M.
\end{align}
The interior Harnack inequality for uniformly elliptic divergence-form
equations with bounded lower-order coefficients
\cite[Theorem~8.20]{GT}
therefore yields
\begin{align}
\sup_{B_1}w
\leqslant
C(n,p,M)\inf_{B_1}w.
\end{align}
Scaling back proves \eqref{eq:crit-harnack}.

For the last assertion, apply \eqref{eq:crit-harnack} with $r=1/2$.
Since $v(0)=1$ and $v\leqslant K$ on $B_1$, condition
\eqref{eq:crit-harnack-potential} holds with a constant depending only on
$n,p,K$. Hence
\begin{align}
1
\leqslant
\sup_{B_{1/2}}v
\leqslant
C\inf_{B_{1/2}}v,
\end{align}
which gives the pointwise lower bound in
\eqref{eq:crit-harnack-mass}; the source-mass bound follows by integration.
\end{proof}

Since $1+\beta>0$ and $2(p+1+\beta)/(p-1)\geqslant n-1$,
testing \eqref{eq:main-equation} with $u^{1+\beta}\eta^2$ and applying
H\"older's and Young's inequalities gives
\begin{align}
\mathcal E(R)
&\leqslant C\mathcal U(2R)+CR^{-2}\int_{B_{2R}}u^{2+\beta}\notag\\*
&\leqslant C\mathcal U(2R)
+CR^{-2+\frac{n(p-1)}{p+1+\beta}}
\mathcal U(2R)^{\frac{2+\beta}{p+1+\beta}}
\leqslant C\mathcal U(2R)+CR.
\label{eq:crit-test}
\end{align}
Since $p+1+\beta>p$, the height bound and Lemma~\ref{lem:crit-source}
imply
\begin{equation}\label{eq:crit-polynomial}
\mathcal U(R)\leqslant CR^{n-2p/(p-1)}.
\end{equation}
If $2p/(p-1)\geqslant n-2$, then \eqref{eq:crit-test} and
\eqref{eq:crit-polynomial} give
$\mathcal E(R)\leqslant CR^2$, and
\eqref{eq:liouville-criterion} completes the proof. This includes all
admissible exponents when $n=5,6$. In the remainder we assume
$n\geqslant7$ and $2/(p-1)<n-4$.

\subsection{Small-defect limits and the critical neck}
For a point \(y\), with natural length
\(\ell=\ell(y)=u(y)^{-(p-1)/2}\), we refer to
\begin{align}\label{09221246}
\ell\int_{B_\ell(y)}D[u]\,dx
\end{align}
as the normalized defect at \(y\).
For fixed $n,p$ in \eqref{eq:liouville-range} and sequences $u_j,y_j$, let
$\ell_j\coloneqq u_j(y_j)^{-(p-1)/2}$ and
$v_j(z)\coloneqq\ell_j^{2/(p-1)}u_j(y_j+\ell_jz)$ as in
\eqref{eq:common-scaled-equation}. The following observation includes
both exponent ranges.

\begin{lemma}\label{lem:crit-small-defect}
Suppose $R_j\to\infty$, $v_j\leqslant K_0$ on $B_{4R_j}$ for a fixed
$K_0>1$, and
\begin{align}\label{09180019}
	\ell_j\int_{B_{\ell_j}(y_j)}\mathcal D[u_j]\,\mathrm dx\longrightarrow0.
\end{align}
Then necessarily $p=(n+2)/(n-2)$, $\ell_j\to\infty$, and, after taking a subsequence,
$v_j\to\mathcal B$ in $C^2_{\mathrm{loc}}(\mathbb R^n)$, where
$-\Delta\mathcal B=\mathcal B^p$, $\mathcal B(0)=1$, and $\mathcal B$
is a standard Euclidean bubble. In particular,
\begin{equation}\label{eq:crit-bubble-mass}
\mathcal B(z)\leqslant C(n,K_0)(1+|z|)^{2-n},\qquad
m_*\coloneqq \int_{\mathbb R^n}\mathcal B^p\,\mathrm dz\in(0,\infty).
\end{equation}
\end{lemma}

\begin{proof}
The bound $v_j\leqslant K_0$, the normalization $v_j(0)=1$, and Harnack's
inequality give positive lower bounds on fixed balls. Testing the scaled
equation \eqref{eq:common-scaled-equation} with $v_j$ times a cutoff function
gives local $L^2$ bounds for $\nabla v_j$.
Differentiating \eqref{eq:common-scaled-equation} gives a uniformly
elliptic equation for each derivative, with bounded potential
$pv_j^{p-1}$. Caccioppoli, Moser, and H\"older estimates followed by
Schauder theory give local $C^{2,\alpha}$ compactness. This argument
also applies on the fixed annuli used below.

After passing to a subsequence, denote the local limit by $v$.
If $\ell_j$ had a finite subsequential limit, it would be positive because
the solutions $u_j$ are uniformly bounded. Using \eqref{09180019}, we would
then obtain
\begin{align}\label{09180020}
	\int_{B_{1}}\mathcal D[v]\,\mathrm dx = 0.
\end{align}
Rescaling \eqref{eq:liouville-defect-estimates} would then force $v$ to be a
constant on $B_1$. Thus $v\equiv v(0)=1$, contrary to its equation.
Thus $\ell_j\to\infty$, and the limit is a positive entire solution
of $-\Delta v=v^p$, with $v(0)=1$.
The subcritical range is excluded by
\cite[Theorem~1.1]{GidasSpruck1981}; at the critical exponent, the
classification in \cite[Corollary~8.2(b), p.~294]{CaffarelliGidasSpruck1989}
applies.
The normalization and upper bound give \eqref{eq:crit-bubble-mass}.
\end{proof}

Define the scaled Green kernel and tail by
\begin{align}
G_j(x,z)&\coloneqq\ell_j^{n-2}G_{u_j}(y_j+\ell_jx,y_j+\ell_jz),\notag\\
\mathcal T_j(r)&\coloneqq r^{n-2}\int_{|z|>r}G_j(0,z)v_j(z)^p\,\mathrm dz.
\label{eq:crit-tail}
\end{align}
Thus, by Lemma~\ref{lem:crit-source},
\begin{align}\label{09181849}
v_j(x)=\int G_j(x,z)v_j(z)^p\,\mathrm dz.
\end{align}
The tail $r^{2-n}\mathcal T_j(r)$ is nonincreasing, but
$\mathcal T_j$ does not have this monotonicity.

\begin{proposition}\label{prop:crit-neck}
Let $p=(n+2)/(n-2)$. Suppose $R_j\to\infty$, $\ell_j\to\infty$,
$v_j\leqslant K_0$ on $B_{4R_j}$, and $v_j\to\mathcal B$ locally in $C^2$,
with $\mathcal B$ as in \eqref{eq:crit-bubble-mass}. Then
$\mathcal T_j(R_j)\to0$.
\end{proposition}

\begin{proof}
We adapt the first-crossing argument of
\cite[proof of Theorem~3.8]{LiLiuMa2026}. That theorem concerns the
CR operator; the changes needed for the present operator are verified
below. In particular, we use the varying kernels $G_j$ and a
nonnegative, rather than conserved, Pohozaev current.

Suppose along a subsequence $\mathcal T_j(R_j)\geqslant\theta_0>0$.
For each fixed $r$, the identity $v_j(0)=1$, the Green representation,
and Lemma~\ref{lem:crit-green-limit} imply
\[
\mathcal T_j(r)\longrightarrow
r^{n-2}\int_{|z|>r}\Phi(z)\mathcal B(z)^p\,\mathrm dz=O(r^{-2}).
\]
Choose $b_j\to\infty$ sufficiently slowly that
\begin{equation}\label{eq:crit-slow-scale}
b_j^2=o(R_j),\quad \mathcal T_j(b_j)\to0,\quad
\|v_j-\mathcal B\|_{C^2(B_{4b_j^2})}\leqslant b_j^{-8n},\quad
\int_{B_{b_j^2}}v_j^p\to m_*.
\end{equation}

\smallskip\noindent
\emph{\textbf{First crossing and annular control.}}
For each fixed $\varepsilon>0$, the doubling lemma
\cite[Lemma~5.1]{PolacikQuittnerSouplet2007}, applied to
$v_j^{2/(n-2)}$ on $B_{r/8}(z)$, and the uniformly elliptic Harnack
inequality give
\begin{equation}\label{eq:crit-tail-detection}
|z|=r\leqslant R_j,\quad r^{(n-2)/2}v_j(z)\geqslant\varepsilon
\quad\Longrightarrow\quad \mathcal T_j(r/4)\geqslant c_\varepsilon>0.
\end{equation}
Choosing the doubling radius constant small in terms of $\varepsilon$
places the selected natural-scale ball inside $B_{r/8}(z)$.
Harnack's inequality gives source mass at least $c_\varepsilon/K_0$ there;
\eqref{eq:crit-green-bounds} proves \eqref{eq:crit-tail-detection}.

Now fix $\varepsilon\leqslant1$ sufficiently small for absorption,
and then choose $0<\theta<\min\{\theta_0,c_\varepsilon,1\}$.
Let $s_j$ be the first radius in $[b_j,R_j]$ at which
$\mathcal T_j(s_j)=\theta$. Then $s_j\to\infty$ and
\begin{equation}\label{eq:crit-neck-smallness}
|z|^{(n-2)/2}v_j(z)<\varepsilon
\quad\text{for }4b_j\leqslant|z|\leqslant s_j.
\end{equation}
Writing
\begin{align}
m_j(r)\coloneqq\int_{B_r}v_j^p\,\mathrm dz,
\end{align}
the annular Harnack inequality and the potential decomposition used in the
cited argument give
\begin{align}
\sup_{r\leqslant|z|\leqslant2r}v_j(z)
&\leqslant Cr^{2-n}\{m_j(r/8)+\theta\},\label{eq:crit-mass-increment1848}\\
m_j(2r)-m_j(r)&\leqslant
Cr^{-2}\{m_j(r/8)+\theta\}^p,
\label{eq:crit-mass-increment}
\end{align}
where $128b_j\leqslant r\leqslant s_j/128$. To obtain
\eqref{eq:crit-mass-increment1848} from \eqref{09181849}, split the potential
into $B_{r/8}$, $B_{16r}\setminus B_{r/8}$, and
$\mathbb R^n\setminus B_{16r}$. Indeed,
\eqref{eq:crit-neck-smallness} bounds $r^2v_j^{p-1}$ by
$C\varepsilon^{p-1}$ on $\{r/16\leqslant|z|\leqslant32r\}$.
A finite Harnack chain inside this larger annulus controls
$\sup_{r/8\leqslant|z|\leqslant16r}v_j$ by
$C\sup_{r\leqslant|z|\leqslant2r}v_j$, uniformly for
$\varepsilon\leqslant1$. The first and third
parts are controlled by \eqref{eq:crit-green-bounds} and
$\mathcal T_j(16r)<\theta$; the middle part is at most
$Cr^2(\sup_{r\leqslant|z|\leqslant2r}v_j)^p$.
Its coefficient relative to this supremum is at most
$C\varepsilon^{p-1}$, so it can be absorbed. The second estimate,
\eqref{eq:crit-mass-increment}, then follows from
\eqref{eq:crit-mass-increment1848}.

Set
\[
r_j^*\coloneqq\begin{cases}
s_j,&s_j<2^{17}b_j,\\
2^{-17}s_j,&s_j\geqslant2^{17}b_j.
\end{cases}
\]

We next prove
\begin{align}\label{09182058}
m_j(r_j^*)\longrightarrow m_*.
\end{align}

If $r_j^*<b_j^2$, then the conclusion follows directly from
\eqref{eq:crit-slow-scale}. Indeed, since $r_j^*\to\infty$ and
$r_j^*<b_j^2$, the $C^2$-convergence of $v_j$ to $\mathcal B$ on
$B_{4b_j^2}$ gives
\[
m_j(r_j^*)
=
\int_{B_{r_j^*}}v_j^p
=
\int_{B_{r_j^*}}\mathcal B^p+o(1)
\longrightarrow
\int_{\mathbb R^n}\mathcal B^p
=
m_*.
\]
It remains to consider the case
\[
r_j^*\geqslant b_j^2.
\]
For all sufficiently large $j$, this necessarily implies
\[
r_j^*=2^{-17}s_j.
\]
Hence, for all sufficiently large $j$, the estimates in
\eqref{eq:crit-mass-increment} apply at every dyadic radius between $b_j^2$
and $2r_j^*$.

Fix the constant
\[
M_*\coloneqq m_*+2.
\]
By \eqref{eq:crit-slow-scale}, for all sufficiently large $j$,
\begin{align}\label{09182052}
	m_j(b_j^2)\leqslant m_*+ 1.
\end{align}

We first claim that
\[
m_j(2^k b_j^2)\leqslant M_*
\]
for every integer $k\geqslant0$ such that $2^k b_j^2\leqslant2r_j^*$.

Suppose, to the contrary, that this fails, and let $k\geqslant1$ be the first index for which
\[
m_j(2^k b_j^2)>M_*.
\]
Then, by the minimality of $k$,
\begin{align}\label{09182053}
	m_j(2^\ell b_j^2)\leqslant M_*,
\qquad
0\leqslant\ell<k.
\end{align}

Since $m_j$ is nondecreasing, for every $0\leqslant\ell<k$ we also have
\[
m_j(2^{\ell-3}b_j^2)
\leqslant
m_j(2^\ell b_j^2)
\leqslant M_*.
\]
Applying \eqref{eq:crit-mass-increment} with $r=2^\ell b_j^2$, we obtain
\[
m_j(2^{\ell+1}b_j^2)
-
m_j(2^\ell b_j^2)
\leqslant
C_{M_*}(2^\ell b_j^2)^{-2},
\]
where $C_{M_*}$ is independent of $j$ and $\ell$.

Summing from $\ell=0$ to $k-1$ gives
\begin{align}\label{09182054}
m_j(2^k b_j^2)-m_j(b_j^2)
&\leqslant
C_{M_*} b_j^{-4}
\sum_{\ell=0}^{k-1}4^{-\ell}
\leqslant
C_{M_*}' b_j^{-4}.
\end{align}

Since $b_j\to\infty$, the right-hand side tends to zero. Combining \eqref{09182052} and \eqref{09182054}, we obtain, for all sufficiently large $j$,
\[
m_j(2^k b_j^2)
\leqslant
m_*+ 1+o(1)
<
m_*+2
=
M_*,
\]
which contradicts the choice of $k$. Hence
\[
m_j(2^k b_j^2)\leqslant M_*
\]
at every relevant dyadic scale.

We may now sum the increment estimate without any further bootstrap. Choose $N_j\geqslant0$ such that
\[
2^{N_j}b_j^2
\leqslant
r_j^*
<
2^{N_j+1}b_j^2.
\]
Since $m_j$ is nondecreasing,
\[
0
\leqslant
m_j(r_j^*)-m_j(b_j^2)
\leqslant
m_j(2^{N_j+1}b_j^2)-m_j(b_j^2).
\]
Using \eqref{eq:crit-mass-increment} at the dyadic radii $2^\ell b_j^2$ and the uniform bound just proved, we obtain
\[
\begin{aligned}
m_j(2^{N_j+1}b_j^2)-m_j(b_j^2)
&\leqslant
C_{M_*} b_j^{-4}
\sum_{\ell=0}^{N_j}4^{-\ell}
\\
&\leqslant
C_{M_*}' b_j^{-4}
=
o(1).
\end{aligned}
\]
Consequently,
\[
m_j(r_j^*)
=
m_j(b_j^2)+o(1).
\]
Finally, using \eqref{eq:crit-slow-scale}, this proves \eqref{09182058}.

\smallskip\noindent
\emph{\textbf{Varying kernels and the regular part.}}
We now pass to the blow-down scale because $s_j\to\infty$. Set
\begin{align}
\widehat v_j(y)\coloneqq s_j^{n-2}v_j(s_jy),\qquad
\widehat G_j(y,z)\coloneqq s_j^{n-2}G_j(s_jy,s_jz).
\end{align}
This scaling leaves the Newton kernel $|x|^{2-n}$ invariant.
On every compact subset of $B_{2^{-18}}\setminus\{0\}$,
\eqref{eq:crit-slow-scale} and the preceding decay estimate imply
\begin{align}
v_j(s_jz)\leqslant Cs_j^{2-n}.
\end{align}
Define the inner source measure
\[
\mathrm d\sigma_j(z)\coloneqq s_j^n v_j(s_jz)^p
\mathbf1_{\{|z|<r_j^*/s_j\}}\,\mathrm dz.
\]

The concentration of the inner source measures follows from
\eqref{09182058} and the decay estimate above. First, the total mass is
preserved:

\begin{align}
\sigma_j(\mathbb R^n)
&=
m_j(r_j^*)
\longrightarrow m_*.
\end{align}

We next show that
\begin{align}
\sigma_j\rightharpoonup m_*\delta_0,
\end{align}
which means no mass remains away from the origin. Let
$\rho>0$ be fixed. For $|z|\geqslant\rho$ and all sufficiently large $j$,
the preceding decay estimate gives
\[
v_j(s_jz)
\leqslant
C|s_jz|^{2-n},
\]
and hence, at the critical exponent
$p=(n+2)/(n-2)$,
\[
s_j^n v_j(s_jz)^p
\leqslant
C s_j^n(s_j|z|)^{-(n+2)}
=
Cs_j^{-2}|z|^{-(n+2)}.
\]
Therefore,
\[
\begin{aligned}
\sigma_j(\{|z|>\rho\})
&\leqslant
Cs_j^{-2}
\int_{|z|>\rho}|z|^{-(n+2)}\,\mathrm dz \\
&\leqslant
Cs_j^{-2}\rho^{-2}
\longrightarrow0.
\end{aligned}
\]
Thus all the mass of $\sigma_j$ concentrates at the origin.

More precisely, let $\varphi\in C_c(\mathbb R^n)$. We write
\[
\begin{aligned}
\int_{\mathbb R^n}\varphi\,\mathrm d\sigma_j
-
m_*\varphi(0)
&=
\int_{\mathbb R^n}
(\varphi(z)-\varphi(0))\,\mathrm d\sigma_j(z)
+
\varphi(0)
\bigl(\sigma_j(\mathbb R^n)-m_*\bigr).
\end{aligned}
\]
The second term tends to zero by the convergence of the total mass.
For the first term, fix $\rho>0$ and split the integral into
$B_\rho$ and its complement:
\[
\begin{aligned}
\left|
\int_{\mathbb R^n}
(\varphi(z)-\varphi(0))\,\mathrm d\sigma_j(z)
\right|
&\leqslant
\sup_{|z|<\rho}
|\varphi(z)-\varphi(0)|
\,\sigma_j(\mathbb R^n)
\\
&\quad
+
2\|\varphi\|_{L^\infty}
\sigma_j(\{|z|\geqslant\rho\}).
\end{aligned}
\]
The second term vanishes as $j\to\infty$, while the first one can be
made arbitrarily small by choosing $\rho$ sufficiently small, since
$\sigma_j(\mathbb R^n)$ is uniformly bounded. Consequently,
\[
\int_{\mathbb R^n}\varphi(z)\,\mathrm d\sigma_j(z)
\longrightarrow
m_*\varphi(0).
\]
Hence,
\[
\sigma_j\rightharpoonup m_*\delta_0.
\]

Define the outer potential by

\begin{align}
\psi_j(y)
&\coloneqq
\widehat v_j(y)
-\int_{\mathbb R^n}\widehat G_j(y,z)\,\mathrm d\sigma_j(z)
\notag\\
&=
\int_{|z|\geqslant r_j^*/s_j}
\widehat G_j(y,z)\,
s_j^n v_j(s_jz)^p\,\mathrm dz>0.
\end{align}
Then 
\begin{align}
	 -\operatorname{div}(a_j\nabla\psi_j)=0,
\qquad\text{in }B_{2^{-18}},
\end{align}
where $a_j(y)\coloneqq W[u_j](y_j+\ell_js_jy)$.
The definition of $r_j^*$ and the first crossing give
\[
\theta\leqslant\psi_j(0)
=\left(\frac{s_j}{r_j^*}\right)^{n-2}\mathcal T_j(r_j^*)
\leqslant2^{17(n-2)}\theta.
\]
Harnack's inequality, weak compactness, and \eqref{eq:crit-coefficients}
therefore
give $\psi_j\to\psi$ locally uniformly on $B_{2^{-19}}$, with
$\Delta\psi=0$ and $\psi(0)\geqslant\theta$.

Lemma~\ref{lem:crit-green-limit} applies at the physical scale
$\ell_js_j\to\infty$. Hence the inner potentials
\begin{align}
\int_{\mathbb R^n}\widehat G_j(y,z)\,\mathrm d\sigma_j(z)
\end{align}
converge locally
uniformly to $m_*\Phi$. This passage also accounts for
the moving pole: on a compact set separated from zero, the contribution
of $|z-y|<\rho$ is at most
$Cs_j^{-2}\int_{|t|<\rho}|t|^{2-n}\,\mathrm dt
\leqslant Cs_j^{-2}\rho^2$.
Outside this neighborhood, uniform kernel convergence and concentration of
$\sigma_j$ at zero give the desired convergence. Thus
\begin{equation}\label{eq:crit-neck-limit}
\widehat v_j\longrightarrow v_*\coloneqq m_*\Phi+\psi
\quad\text{locally uniformly on }B_{2^{-19}}\setminus\{0\},
\qquad \psi(0)>0.
\end{equation}

\smallskip\noindent
\emph{\textbf{Boundary convergence and the sign contradiction.}}
The rescaled equation and Lorentz factor are
\begin{equation}\label{eq:crit-neck-equation}
-\operatorname{div}(a_j\nabla\widehat v_j)
=s_j^{-2}\widehat v_j^p.
\end{equation}
Here $a_j=W[u_j](y_j+\ell_js_j\,\cdot)\leqslant C$.
The equations remain uniformly elliptic, with bounded
potential $s_j^{-2}p\widehat v_j^{p-1}$ on fixed punctured compact
sets. The estimates in Lemma~\ref{lem:crit-small-defect} therefore
upgrade \eqref{eq:crit-neck-limit} to local $C^1$ convergence;
the displayed formula also gives $a_j\to1$ there.

The associated Pohozaev current can be written directly as
\begin{align*}
\mathcal P_j\coloneqq{}&a_j
\left[\left(y\cdot\nabla\widehat v_j+\frac{n-2}{2}\widehat v_j\right)
\nabla\widehat v_j
-\frac{y|\nabla\widehat v_j|^2}{1+ a_j}\right] +\frac{s_j^{-2}}{p+1}y\widehat v_j^{p+1}.
\end{align*}
The critical relation $n/(p+1)=(n-2)/2$ cancels the reaction terms.
Direct differentiation gives
\begin{equation}\label{eq:crit-pohozaev-sign}
\operatorname{div}\mathcal P_j
=\frac n2\frac{a_j(a_j-1)}{a_j+1}
|\nabla\widehat v_j|^2\geqslant0.
\end{equation}
Each $\widehat v_j$ is smooth at the origin. Integrate
\eqref{eq:crit-pohozaev-sign} on a fixed $B_r$, $0<r<2^{-20}$,
and pass to the limit only on its boundary. With $\nu_{\mathrm E}$ the
outward Euclidean unit normal, the local $C^1$ convergence gives
\[
\int_{\partial B_r}
\left[
\left(y\cdot\nabla v_*+\frac{n-2}{2}v_*\right)\nabla v_*
-\frac12y|\nabla v_*|^2
\right]\cdot\nu_{\mathrm E}\,\mathrm dS\geqslant0.
\]
The following identity is inspired by
\cite[Proposition~4.3, (4.18)]{Zhang2026}. Applying
Lemma~\ref{09201546} to
\[
v_*=m_*\Phi+\psi
\]
gives
\[
\int_{\partial B_r}
\left[
\left(
y\cdot\nabla v_*+\frac{n-2}{2}v_*
\right)\nabla v_*
-
\frac12 y|\nabla v_*|^2
\right]\cdot\nu_{\mathrm E}\,\mathrm dS
=
-\frac{n-2}{2}m_*\psi(0)<0.
\]
This contradicts the nonnegativity obtained above. Hence the assumption
$\mathcal T_j(R_j)\geqslant\theta_0$ is impossible, and therefore
\[
\mathcal T_j(R_j)\longrightarrow0.
\]
\end{proof}

\begin{lemma}\label{09201546}

Let $n\geqslant3$, and let
\[
\Phi(x)
=
\frac{|x|^{2-n}}
{(n-2)|\mathbb S^{n-1}|}.
\]
Suppose
\[
v(x)=m\Phi(x)+\psi(x)
\]
in a punctured ball $B_{r_0}\setminus\{0\}$, where $m>0$ and $\psi$ is
harmonic in $B_{r_0}$. Define
\[
\mathcal P(r;v)
\coloneqq
\int_{\partial B_r}
\left[
\left(
x\cdot\nabla v+\frac{n-2}{2}v
\right)\nabla v
-
\frac12 x|\nabla v|^2
\right]\cdot\nu_{\mathrm E}\,\mathrm dS.
\]
Then, for every $0<r<r_0$,
\[
\mathcal P(r;v)
=
-\frac{n-2}{2}m\psi(0).
\]
\end{lemma}
\begin{proof}
On $\partial B_r$ one has $x=r\nu_{\mathrm E}$, and hence
\[
x\cdot\nabla v=r\partial_rv.
\]
Therefore
\begin{align}\label{09201602}
\mathcal P(r;v)
=
\int_{\partial B_r}
\left[
\left(
r\partial_rv+\frac{n-2}{2}v
\right)\partial_rv
-
\frac r2|\nabla v|^2
\right]\,\mathrm dS.
\end{align}

We first record the contribution of the singular term $m\Phi$. Since $\Phi$
is radial and
\begin{align}\label{09201608}
r\partial_r\Phi=-(n-2)\Phi,
\end{align}
we have
\[
\begin{aligned}
\mathcal P(r;m\Phi)&=\int_{\partial B_r}\Bigg[
-\frac{n-2}{2}m\Phi\,\partial_r(m\Phi)
-
\frac r2\bigl|\partial_r(m\Phi)\bigr|^2 \Bigg]\,\mathrm dS
=
0.
\end{aligned}
\]
Thus the pure $m\Phi$ contribution to $\mathcal P(r;v)$ vanishes.
Similarly, the pure $\psi$ contribution also vanishes.
It remains to compute the mixed terms. Substituting $v=m\Phi+\psi$ into
\eqref{09201602}, using the radiality of $\Phi$, and subtracting the two pure
contributions above, we obtain
\[
\begin{aligned}
\mathcal P(r;v)
&=
m\int_{\partial B_r}
\left[
r\partial_r\Phi\,\partial_r\psi
+
\frac{n-2}{2}\Phi\,\partial_r\psi
+
\frac{n-2}{2}\psi\,\partial_r\Phi
\right]\,\mathrm dS.
\end{aligned}
\]
Using \eqref{09201608},
this becomes
\begin{align}\label{0920c}
\mathcal P(r;v)
=
\frac{n-2}{2}m
\int_{\partial B_r}
\left(
-\Phi\,\partial_r\psi
+
\psi\,\partial_r\Phi
\right)\,\mathrm dS.
\end{align}

Because $\Phi$ is constant on $\partial B_r$ and $\psi$ is harmonic,
\begin{align}\label{0920d}
\int_{\partial B_r}\Phi\,\partial_r\psi\,\mathrm dS
=
\Phi(r)
\int_{\partial B_r}\partial_r\psi\,\mathrm dS
=
\Phi(r)
\int_{B_r}\Delta\psi\,\mathrm dx
=
0.
\end{align}
Moreover,
\[
\partial_r\Phi
=
-\frac{1}{|\mathbb S^{n-1}|}r^{1-n},
\]
while the mean-value property for $\psi$ gives
\[
\int_{\partial B_r}\psi\,\mathrm dS
=
|\mathbb S^{n-1}|r^{n-1}\psi(0).
\]
Consequently,
\begin{align}\label{0920e}
	\int_{\partial B_r}
\psi\,\partial_r\Phi\,\mathrm dS
=
-\psi(0).
\end{align}
Substituting \eqref{0920d} and \eqref{0920e} into \eqref{0920c} proves the
lemma.
\end{proof}

\begin{corollary}\label{cor:crit-neck-cost}
Fix $n,p$ as in \eqref{eq:liouville-range}, and let $K_0>1$ and
$\theta_0>0$. Then there exist constants
\[
R_*=R_*(n,p,K_0,\theta_0)>0,
\qquad
\delta_*=\delta_*(n,p,K_0,\theta_0)>0
\]
such that the following holds.

Let $u$ be a solution in \eqref{eq:liouville-range}, let $y\in\mathbb R^n$,
and set
\[
\ell\coloneqq u(y)^{-(p-1)/2},
\qquad
v(z)\coloneqq
\ell^{2/(p-1)}u(y+\ell z).
\]
Define the normalized tail $\mathcal T$ as in \eqref{eq:crit-tail}.
If
\[
r_0\geqslant R_*,
\qquad
v\leqslant K_0 \quad\text{on } B_{4r_0},
\]
and, in addition, either
\[
p<\frac{n+2}{n-2},
\]
or
\[
p=\frac{n+2}{n-2}
\quad\text{and}\quad
\mathcal T(r_0)\geqslant\theta_0,
\]
then
\begin{equation}\label{eq:crit-neck-cost}
\ell\int_{B_\ell(y)}\mathcal D[u]\,\mathrm dx
\geqslant
\delta_*.
\end{equation}
\end{corollary}

\begin{proof}
Suppose the conclusion were false for the fixed parameters
$n,p,K_0,\theta_0$. Then, for every $j\geqslant1$, there would exist a solution
$u_j$, a point $y_j\in\mathbb R^n$, and a radius $r_{0,j}\geqslant j$ such that,
with
\[
\ell_j\coloneqq u_j(y_j)^{-(p-1)/2},
\qquad
v_j(z)\coloneqq
\ell_j^{2/(p-1)}u_j(y_j+\ell_j z),
\]
one has
\[
v_j\leqslant K_0
\quad\text{on }B_{4r_{0,j}},
\]
together with
\[
\ell_j
\int_{B_{\ell_j}(y_j)}
\mathcal D[u_j]\,\mathrm dx
<
\frac1j.
\]
In particular,
\[
r_{0,j}\longrightarrow\infty,
\qquad
\ell_j
\int_{B_{\ell_j}(y_j)}
\mathcal D[u_j]\,\mathrm dx
\longrightarrow0.
\]

Applying Lemma~\ref{lem:crit-small-defect} with
$R_j=r_{0,j}$, after passing to a subsequence we obtain
\[
p=\frac{n+2}{n-2},
\qquad
\ell_j\longrightarrow\infty,
\]
and
\[
v_j\longrightarrow\mathcal B
\quad\text{in }C^2_{\mathrm{loc}}(\mathbb R^n),
\]
where $\mathcal B$ is a standard critical bubble.

Hence the subcritical alternative
\[
p<\frac{n+2}{n-2}
\]
is already impossible. We are therefore in the critical case, and the
hypothesis of the corollary gives
\[
\mathcal T_j(r_{0,j})\geqslant\theta_0
\]
for every $j$. Proposition~\ref{prop:crit-neck}, applied with
$R_j=r_{0,j}$, yields
\[
\mathcal T_j(r_{0,j})\longrightarrow0,
\]
which contradicts the preceding lower bound. This proves the corollary.
\end{proof}

\subsection{Separated cores and the counting estimate}

We return to the full range \eqref{eq:liouville-range}, subject to
$n\geqslant7$ and $2/(p-1)<n-4$. All constants in the core construction
are fixed before the spatial radius $R$ is chosen. 
Set
\begin{align}
K_0\coloneqq4000^{2/(p-1)}.
\label{eq:core-K0}
\end{align}
The numerical choice of $4000$ ensures the scale and distance contractions
used in Lemma~\ref{lem:crit-scale-drop}.
Then 
\begin{align}\label{09202309}
K_0\geqslant(5/4)^{2/(p-1)},
\end{align}
and
\begin{align}
2K_0^{-(p-1)/2}\leqslant\frac14,
\qquad
\frac1{1000}+103K_0^{-(p-1)/2}<\frac1{20}.
\label{eq:core-K0-small}
\end{align}

At the critical exponent, consider any normalized solution $v$ arising from
the natural scaling \eqref{09201951}, with $v(0)=1$ and $v\leqslant K_0$ on
a fixed ball. Uniform ellipticity and Harnack's inequality give a constant
$c=c(n,K_0)>0$ such that
\begin{align}
\int_{B_{1/2}}v^p\,\mathrm dz\geqslant c.
\label{eq:normalized-source-mass}
\end{align}
Combining this constant with the lower Green-kernel bound in \eqref{eq:crit-green-bounds}, we fix
\begin{align}
\theta_0\coloneqq c\,4000^{-(n-2)}>0
\label{eq:theta0-choice}
\end{align}
in the critical case, with $c$ decreased if necessary. In the subcritical
case, the tail alternative in Corollary~\ref{cor:crit-neck-cost} is not used,
and we simply set $\theta_0\coloneqq1$.

For both exponent ranges, let $R_*$ and $\delta_*$ be the constants given by Corollary~\ref{cor:crit-neck-cost} for the above choices of $K_0$ and $\theta_0$, and set
\begin{align}
\Lambda\coloneqq\max\{10,4000R_*\}.
\label{eq:Lambda-choice}
\end{align}

Use the contact-envelope construction of
\cite[Lemma~4.4 and (4.19)--(4.22)]{LiLiuMa2026}, with Euclidean
distance and natural length \eqref{eq:crit-natural-length}:
\begin{equation}\label{eq:crit-envelope}
\underline\ell(x)\coloneqq\inf_{y\in\mathbb R^n}
\left\{\ell(y)+\frac{|x-y|}{100\Lambda}\right\}.
\end{equation}
This part of the construction uses only the triangle inequality,
properness and the positive lower bound for $\ell$.

For every $x\in\mathbb R^n$, the infimum is attained. Let $y_x$ be a minimizer for $\underline\ell(x)$, so that
\begin{align}
\underline\ell(x)
=
\ell(y_x)+\frac{|x-y_x|}{100\Lambda}.
\label{eq:crit-contact-minimizer}
\end{align}
Then $y_x$ is a contact point. Indeed, if $\underline\ell(y_x)<\ell(y_x)$, then some $z\in\mathbb R^n$ would satisfy
\begin{align}
\ell(z)+\frac{|y_x-z|}{100\Lambda}<\ell(y_x),
\end{align}
and the triangle inequality would give
\begin{align}
\ell(z)+\frac{|x-z|}{100\Lambda}
<
\ell(y_x)+\frac{|x-y_x|}{100\Lambda}
=
\underline\ell(x),
\end{align}
contradicting the minimality of $y_x$. Hence
\begin{align}
\underline\ell(y_x)=\ell(y_x).
\label{eq:crit-contact-point}
\end{align}
Moreover, for every $z\in\mathbb R^n$,
\begin{align}
\ell(z)
\geqslant
\ell(y_x)-\frac{|z-y_x|}{100\Lambda}.
\label{eq:crit-contact-lower}
\end{align}
In particular, $\underline\ell\leqslant\ell$ on $\mathbb R^n$.
Consequently, using \eqref{09202309}, every contact point $y$ satisfies
\begin{equation}\label{eq:crit-contact-height}
u(z)\leqslant (5/4)^{2/(p-1)}\ell(y)^{-2/(p-1)} \leqslant K_0\ell(y)^{-2/(p-1)}
\quad\text{on }B_{20\Lambda\ell(y)}(y).
\end{equation}
Let $\mathcal A$ denote the set of all contact points. We now extract from
$\mathcal A$ a Vitali subfamily that is separated at the natural scale.
Decompose $\mathcal A$ into the dyadic length classes
\begin{align}
\mathcal A_m
\coloneqq
\left\{
y\in\mathcal A:
2^m\leqslant \ell(y)<2^{m+1}
\right\},
\qquad m\in\mathbb Z.
\label{eq:crit-dyadic-contact}
\end{align}
Starting from the smallest nonempty dyadic length classes and proceeding to larger ones, we keep adding contact points as long as the separation condition remains valid, and stop in each class only when no further point from that class can be added:
\begin{align}
|x_i-x_j|
\geqslant
\Lambda(\ell_i+\ell_j)
\qquad
(i\neq j),
\label{eq:crit-core-separation}
\end{align}
where $\ell_i\coloneqq\ell(x_i)$. Denote the resulting selected family by
\begin{align}
\{x_i:i\in\mathcal C\}\subset\mathcal A.
\label{eq:crit-core-family}
\end{align}
Thus $\mathcal C$ is the index set of the selected contact points. We call each $i\in\mathcal C$ a core, with center $x_i$ and natural length $\ell_i$.

The maximality of the selection gives the corresponding covering property: for every contact point $y\in\mathcal A$, there exists at least one $i\in\mathcal C$ such that
\begin{align}
\ell_i
&\leqslant
2\ell(y),
&
|x_i-y|
&\leqslant
3\Lambda\ell(y).
\label{eq:crit-core-cover}
\end{align}
Indeed, when the dyadic class $\mathcal A_m$ containing $y$ is processed, either $y$ is selected or it is excluded by a previously selected point $x_i$. In the latter case, the order of the construction gives $\ell_i\leqslant2\ell(y)$, while failure of the separation condition gives
\begin{align}
|x_i-y|
<
\Lambda(\ell_i+\ell(y))
\leqslant
3\Lambda\ell(y).
\end{align}

By Lemma~\ref{lem:height-bound} and the definition
\eqref{eq:crit-natural-length}, $\ell$ has a positive global lower bound, and
the separation property implies that the selected family is locally finite.
Moreover, the balls $B_{\Lambda\ell_i}(x_i)$ are pairwise disjoint.

A core \(i\) is called good if its normalized defect which is defined in \eqref{09221246} satisfies
\begin{align}\label{eq:crit-good-core}
\ell_i\int_{B_{\ell_i}(x_i)}D\,dx<\delta_*,
\end{align}
and bad otherwise. Here $\mathcal D$ is defined in
\eqref{eq:liouville-current}. By Corollary~\ref{cor:crit-neck-cost} and
\eqref{eq:crit-contact-height}, applied with $r_0\coloneqq5\Lambda$, every
subcritical core is bad. Hence good cores can occur only at the critical
exponent.

At this point the two exponent ranges separate only in the treatment of
good cores. In the subcritical case, Corollary~\ref{cor:crit-neck-cost}
and \eqref{eq:crit-contact-height} show that every core is bad. Hence no
further structure of good cores is needed in that case.

The remainder of this subsection concerning good cores, beginning with
Lemma~\ref{lem:crit-scale-drop}, is therefore relevant only at the critical
exponent
\begin{align}
p=\frac{n+2}{n-2}.
\label{eq:crit-endpoint-for-good-cores}
\end{align}
Its purpose is to control the possible critical good cores by assigning them
to bad cores through a scale-decreasing parent--descendant structure.

After this critical counting step, the two exponent ranges are combined
again: the bad-core estimate and the integral feedback in the next
subsection use only the natural-length identity, the inequality
\begin{align}
\frac{2(p+1+\beta)}{p-1}\geqslant n-1,
\label{eq:crit-common-exponent}
\end{align}
and the assumption
\begin{align}
p\leqslant\frac{n+2}{n-2},
\end{align}
and therefore apply uniformly in both the subcritical and critical ranges.

\begin{lemma}\label{lem:crit-scale-drop}
If $i$ is good, $j\ne i$ and $\ell_i\leqslant\ell_j$, then there is
a core $k$ such that
\begin{equation}\label{eq:crit-scale-drop}
\ell_k\leqslant\frac14\ell_i,\qquad
|x_i-x_k|\leqslant\frac1{20}|x_i-x_j|.
\end{equation}
\end{lemma}

\begin{proof}
Since $i$ is good and every subcritical core is bad, we must have
\begin{align}
p=\frac{n+2}{n-2}.
\label{eq:crit-good-implies-critical}
\end{align}
In particular,
\begin{align}
\ell(x)=u(x)^{-2/(n-2)}.
\label{eq:crit-length-critical}
\end{align}
Set
\begin{align}
d\coloneqq|x_i-x_j|.
\label{eq:crit-core-distance}
\end{align}

We first record the amount of source carried by the core $j$. Normalize at
$x_j$ by
\begin{align}
v_j(\xi)\coloneqq\ell_j^{(n-2)/2}u(x_j+\ell_j\xi).
\label{eq:crit-normalization-j}
\end{align}
Then $v_j(0)=1$. Since $x_j$ is a contact point,
\eqref{eq:crit-contact-height} and the choice of $K_0$ give a uniform
upper bound for $v_j$ on a fixed ball. Hence the scale-invariant Harnack
inequality, Lemma~\ref{lem:crit-harnack}, gives
\begin{align}
v_j\geqslant c
\qquad\text{on }B_{1/2},
\label{eq:crit-vj-lower}
\end{align}
where $c>0$ depends only on the fixed parameters. Therefore
\begin{align}
\int_{B_{1/2}}v_j^p\,\mathrm d\xi
\geqslant c.
\label{eq:crit-vj-source}
\end{align}
Scaling back and using
$p=(n+2)/(n-2)$, we obtain
\begin{align}
\int_{B_{\ell_j/2}(x_j)}u^p\,\mathrm dx
&=
\ell_j^{\,n-p(n-2)/2}
\int_{B_{1/2}}v_j^p\,\mathrm d\xi
\nonumber\\
&=
\ell_j^{(n-2)/2}
\int_{B_{1/2}}v_j^p\,\mathrm d\xi
\nonumber\\
&\geqslant
c\,\ell_j^{(n-2)/2}.
\label{eq:crit-core-j-source}
\end{align}

We claim that there exists
$z\in B_{d/1000}(x_i)$ such that
\begin{align}
u(z)>K_0\ell_i^{-(n-2)/2}.
\label{eq:crit-high-point}
\end{align}
Suppose, to the contrary, that
\begin{align}
u\leqslant
K_0\ell_i^{-(n-2)/2}
\qquad\text{throughout }B_{d/1000}(x_i).
\label{eq:crit-no-high-point}
\end{align}

Normalize now at the core $i$:
\begin{align}
v_i(\xi)\coloneqq\ell_i^{(n-2)/2}u(x_i+\ell_i\xi),
\qquad
v_i(0)=1,
\label{eq:crit-normalization-i}
\end{align}
and set
\begin{align}
r_0\coloneqq\frac{d}{4000\ell_i}.
\label{eq:crit-r0}
\end{align}
Since
$4r_0\ell_i=d/1000$, assumption
\eqref{eq:crit-no-high-point} becomes
\begin{align}
v_i\leqslant K_0
\qquad\text{on }B_{4r_0}.
\label{eq:crit-normalized-bound-i}
\end{align}

Moreover, the core separation property gives
\begin{align}
d
\geqslant
\Lambda(\ell_i+\ell_j).
\label{eq:crit-separation-used}
\end{align}
Since $\Lambda\geqslant4000R_*$, we have
\begin{align}
r_0
=
\frac{d}{4000\ell_i}
\geqslant
\frac{\Lambda(\ell_i+\ell_j)}{4000\ell_i}
\geqslant
\frac{\Lambda}{4000}
\geqslant
R_*.
\label{eq:crit-r0-lower}
\end{align}

We next show that the source carried by the core $j$ produces a definite
contribution to the normalized tail of the core $i$. By
\eqref{eq:crit-separation-used},
\begin{align}
\ell_j\leqslant\frac{d}{\Lambda}.
\label{eq:crit-ellj-vs-d}
\end{align}
Since $\Lambda\geqslant10$, every
$x\in B_{\ell_j/2}(x_j)$ satisfies
\begin{align}
|x-x_i|
&\geqslant
d-\frac{\ell_j}{2}
\geqslant
\left(1-\frac1{2\Lambda}\right)d
\geqslant
\frac{19}{20}d,
\label{eq:crit-jball-distance-lower}\\
|x-x_i|
&\leqslant
d+\frac{\ell_j}{2}
\leqslant
\left(1+\frac1{2\Lambda}\right)d
\leqslant
\frac{21}{20}d.
\label{eq:crit-jball-distance-upper}
\end{align}
In particular,
\begin{align}
B_{\ell_j/2}(x_j)
\subset
\mathbb R^n\setminus B_{d/4000}(x_i).
\label{eq:crit-jball-outside}
\end{align}
Thus, after the normalization at $x_i$, the whole image of
$B_{\ell_j/2}(x_j)$ lies in the region
$\{|\xi|>r_0\}$ entering the definition of the tail $\mathcal T_i(r_0)$.

Let
\begin{align}
\Omega_{ij}
\coloneqq
\left\{
\xi\in\mathbb R^n:
x_i+\ell_i\xi\in B_{\ell_j/2}(x_j)
\right\}.
\label{eq:crit-Omegaij}
\end{align}
Then $\Omega_{ij}\subset\{|\xi|>r_0\}$.
The scaled Green kernel associated with the normalization at $x_i$ is
\begin{align}
G_i(\xi,\eta)
\coloneqq
\ell_i^{n-2}
G_u(x_i+\ell_i\xi,x_i+\ell_i\eta).
\label{eq:crit-Gi}
\end{align}
By \eqref{eq:crit-green-bounds} and
\eqref{eq:crit-jball-distance-upper}, for every
$\xi\in\Omega_{ij}$,
\begin{align}
G_i(0,\xi)
\geqslant
c|\xi|^{2-n}
\geqslant
c\left(\frac{d}{\ell_i}\right)^{2-n}.
\label{eq:crit-Gi-lower-on-j}
\end{align}

We also compute the source mass of $\Omega_{ij}$ in the $i$-normalization.
Since
$v_i(\xi)=\ell_i^{(n-2)/2}u(x_i+\ell_i\xi)$,
the change of variables
$x=x_i+\ell_i\xi$ gives
\begin{align}
\int_{\Omega_{ij}}v_i^p\,\mathrm d\xi
&=
\ell_i^{p(n-2)/2-n}
\int_{B_{\ell_j/2}(x_j)}u^p\,\mathrm dx
\nonumber\\
&=
\ell_i^{-(n-2)/2}
\int_{B_{\ell_j/2}(x_j)}u^p\,\mathrm dx
\nonumber\\
&\geqslant
c
\left(\frac{\ell_j}{\ell_i}\right)^{(n-2)/2},
\label{eq:crit-scaled-j-source}
\end{align}
where we used \eqref{eq:crit-core-j-source} in the last step.

Keeping only the contribution of $\Omega_{ij}$ in the definition of the
normalized tail, and using
\eqref{eq:crit-Gi-lower-on-j} and
\eqref{eq:crit-scaled-j-source}, we obtain
\begin{align}
\mathcal T_i(r_0)
&=
r_0^{n-2}
\int_{|\xi|>r_0}
G_i(0,\xi)v_i(\xi)^p\,\mathrm d\xi
\nonumber\\
&\geqslant
r_0^{n-2}
\int_{\Omega_{ij}}
G_i(0,\xi)v_i(\xi)^p\,\mathrm d\xi
\nonumber\\
&\geqslant
c\,r_0^{n-2}
\left(\frac{d}{\ell_i}\right)^{2-n}
\left(\frac{\ell_j}{\ell_i}\right)^{(n-2)/2}.
\label{eq:crit-tail-before-cancel}
\end{align}
By the definition of $r_0$,
\begin{align}
r_0^{n-2}
\left(\frac{d}{\ell_i}\right)^{2-n}
=
4000^{-(n-2)}.
\label{eq:crit-tail-cancellation}
\end{align}
Consequently,
\begin{align}
\mathcal T_i(r_0)
\geqslant
c\,4000^{-(n-2)}
\left(\frac{\ell_j}{\ell_i}\right)^{(n-2)/2}
\geqslant
\theta_0,
\label{eq:crit-tail-lower}
\end{align}
where the last inequality follows from $\ell_i\leqslant\ell_j$ and the
choice of $\theta_0$.

We have now verified all the hypotheses of
Corollary~\ref{cor:crit-neck-cost} for the normalization at $x_i$:
\eqref{eq:crit-r0-lower} gives $r_0\geqslant R_*$,
\eqref{eq:crit-normalized-bound-i} gives
$v_i\leqslant K_0$ on $B_{4r_0}$, and
\eqref{eq:crit-tail-lower} gives
$\mathcal T_i(r_0)\geqslant\theta_0$. Hence
\begin{align}
\ell_i
\int_{B_{\ell_i}(x_i)}
\mathcal D\,\mathrm dx
\geqslant
\delta_*.
\label{eq:crit-i-cost}
\end{align}
This contradicts the assumption that $i$ is good. Therefore
\eqref{eq:crit-high-point} holds.

Choose $z\in B_{d/1000}(x_i)$ satisfying
\eqref{eq:crit-high-point}. By
\eqref{eq:crit-length-critical},
\begin{align}
\ell(z)
=
u(z)^{-2/(n-2)}
<
K_0^{-2/(n-2)}\ell_i.
\label{eq:crit-small-length-z}
\end{align}

Let $y$ be a minimizer in the definition of $\underline\ell(z)$. Then $y$
is a contact point and
\begin{align}
\underline\ell(z)
=
\ell(y)+\frac{|z-y|}{100\Lambda}
\leqslant
\ell(z).
\label{eq:crit-contact-from-z}
\end{align}
It follows that
\begin{align}
\ell(y)
&\leqslant
\ell(z)
<
K_0^{-2/(n-2)}\ell_i,
\label{eq:crit-y-length}\\
|z-y|
&\leqslant
100\Lambda\,\ell(z)
<
100\Lambda K_0^{-2/(n-2)}\ell_i.
\label{eq:crit-y-distance}
\end{align}

Apply the covering property \eqref{eq:crit-core-cover} to the contact point
$y$. There exists a core $k$ such that
\begin{align}
\ell_k
&\leqslant
2\ell(y),
&
|x_k-y|
&\leqslant
3\Lambda\ell(y).
\label{eq:crit-cover-y-by-k}
\end{align}
Combining this with \eqref{eq:crit-y-length} gives
\begin{align}
\ell_k
<
2K_0^{-2/(n-2)}\ell_i.
\label{eq:crit-k-length-pre}
\end{align}
Moreover, by the triangle inequality,
\begin{align}
|x_i-x_k|
&\leqslant
|x_i-z|+|z-y|+|y-x_k|
\nonumber\\
&<
\frac{d}{1000}
+
103\Lambda K_0^{-2/(n-2)}\ell_i.
\label{eq:crit-k-distance-1}
\end{align}
Since
$d\geqslant\Lambda\ell_i$, this yields
\begin{align}
|x_i-x_k|
<
\left(
\frac1{1000}
+
103K_0^{-2/(n-2)}
\right)d.
\label{eq:crit-k-distance-pre}
\end{align}

Finally, at the critical exponent,
\begin{align}
\frac{2}{n-2}
=
\frac{p-1}{2}.
\label{eq:crit-exponent-relation}
\end{align}
Hence
$K_0^{-2/(n-2)}=K_0^{-(p-1)/2}$, and
\eqref{eq:core-K0} gives
\begin{align}
\ell_k
\leqslant
\frac14\ell_i,
\qquad
|x_i-x_k|
\leqslant
\frac1{20}d
=
\frac1{20}|x_i-x_j|.
\end{align}
This proves the lemma.
\end{proof}

\medskip
\noindent
The role of Lemma~\ref{lem:crit-scale-drop} is to provide the scale-drop
mechanism needed for
the counting argument below. Whenever a good core $i$ can be compared
with a core $j$ of no smaller natural length, the lemma produces another
core $k$ satisfying
\begin{align}
\ell_k
&\leqslant
\frac14\ell_i,
&
|x_i-x_k|
&\leqslant
\frac1{20}|x_i-x_j|.
\label{eq:crit-core-scale-drop}
\end{align}
Thus, whenever this construction can be continued, one obtains a chain of
cores
\begin{align}
i_0\longrightarrow i_1\longrightarrow i_2\longrightarrow\cdots
\end{align}
whose natural lengths decrease geometrically,
\begin{align}
\ell_{i_{m+1}}
\leqslant
\frac14\ell_{i_m}.
\label{eq:crit-core-chain-scale}
\end{align}
The chain cannot continue indefinitely. Indeed, Lemma~\ref{lem:height-bound} and
\eqref{eq:crit-natural-length} give a uniform positive lower bound
$\ell\geqslant c>0$ on $\mathbb R^n$. Hence a sequence satisfying
\eqref{eq:crit-core-chain-scale} must terminate after finitely many steps.
Moreover, the second estimate in \eqref{eq:crit-core-scale-drop} shows that
the successive spatial displacements decrease geometrically as well, so the
whole chain remains localized near its initial core.

This observation is the basis of the parent--descendant structure used in
Lemma~\ref{lem:crit-count}: good cores can be followed toward strictly smaller
scales until the chain reaches a bad core, apart from the exceptional terminal
cores treated in the localized counting argument.

For $R\geqslant2$, define the finite sets and their cardinalities by
\begin{align*}
\mathcal C_R
&\coloneqq\{i\in\mathcal C:|x_i|\leqslant(1+103\Lambda)R,\ \ell_i\leqslant2R\},
&N_R&\coloneqq\#\mathcal C_R,\\
\mathcal C_R^{\mathrm{bad}}
&\coloneqq\{i\in\mathcal C:i\text{ is bad},\
|x_i|\leqslant2(1+103\Lambda)R,\ \ell_i\leqslant2R\},
&N_R^{\mathrm{bad}}&\coloneqq\#\mathcal C_R^{\mathrm{bad}}.
\end{align*}
The sets are tailored to the single threshold $\ell=R$ used below.

\begin{lemma}\label{lem:crit-count}
There exists a constant $C_0>1$ such that, for every $R\geqslant2$,
\begin{equation}\label{eq:crit-count}
N_R\leqslant C\{1+\log(e+R)N_R^{\mathrm{bad}}\},\qquad
N_R^{\mathrm{bad}}\leqslant
C[R\mathfrak D(C_0R)]^{n/(n+1)}.
\end{equation}
Here $\mathfrak D$ is defined in \eqref{eq:liouville-integrals}.
\end{lemma}

\begin{proof}
We first prove the estimate for $N_R$.

If
\begin{align}
p<\frac{n+2}{n-2},
\end{align}
then every core is bad by Corollary~\ref{cor:crit-neck-cost} and
\eqref{eq:crit-contact-height}. Hence
\begin{align}
N_R\leqslant N_R^{\mathrm{bad}},
\end{align}
and the first estimate follows immediately.

It remains to treat the critical exponent
\begin{align}
p=\frac{n+2}{n-2},
\end{align}
where good cores may occur. We use the scale-drop mechanism of
Lemma~\ref{lem:crit-scale-drop} to assign all but finitely many exceptional
good cores to bad ancestors.

For each core $i$, define the set of admissible parents by
\begin{align}
\mathfrak P(i)
\coloneqq
\left\{
k\in\mathcal C:
\ell_k\leqslant\frac14\ell_i
\right\}.
\label{eq:crit-admissible-parents}
\end{align}
If $\mathfrak P(i)\neq\varnothing$, choose and fix one
$\mathfrak p(i)\in\mathfrak P(i)$ such that
\begin{align}
|x_i-x_{\mathfrak p(i)}|
=
\min_{k\in\mathfrak P(i)}
|x_i-x_k|,
\label{eq:crit-parent-choice}
\end{align}
and set
\begin{align}
e_i\coloneqq|x_i-x_{\mathfrak p(i)}|.
\label{eq:crit-parent-edge}
\end{align}
If the minimum is attained by more than one core, choose one of them
arbitrarily and keep this choice fixed throughout the argument. No
good--bad condition is imposed on the parent.
If $\mathfrak p(j)=i$, we call $j$ a child of $i$. If
$\mathfrak P(i)=\varnothing$, we call $i$ a root.
The closest admissible parent exists whenever the admissible set is nonempty.
Indeed, after fixing one admissible core, any minimizing sequence for the
distance from $x_i$ stays in a bounded ball, and the local finiteness of the
core family leaves only finitely many candidates in that ball.

We record two consequences of Lemma~\ref{lem:crit-scale-drop}. First, suppose
that $i$ is good and has a child $j$, so that $\mathfrak p(j)=i$. Then
$\ell_i\leqslant\ell_j/4<\ell_j$. Applying
Lemma~\ref{lem:crit-scale-drop} to the pair $(i,j)$
produces a core $k$ such that
\begin{align}
\ell_k
\leqslant
\frac14\ell_i,
\qquad
|x_i-x_k|
\leqslant
\frac1{20}|x_i-x_j|
=
\frac1{20}e_j.
\end{align}
Since $\mathfrak p(i)$ is the closest core satisfying
\eqref{eq:crit-parent-choice}, we obtain
\begin{align}
e_i\leqslant\frac1{20}e_j.
\label{eq:crit-parent-edge-drop}
\end{align}
In particular, by Lemma~\ref{lem:crit-scale-drop}, a good root cannot have a
child. Moreover, there is at most one good root: if $i$ and $j$ were two
distinct good roots and $\ell_i\leqslant\ell_j$,
Lemma~\ref{lem:crit-scale-drop} would produce a core of length at most
$\ell_i/4$, contradicting the fact that $i$ is a root.

Second, every parent chain is finite. Indeed, along such a chain the natural
length decreases by at least a factor $1/4$ at each step, while Lemma~\ref{lem:height-bound}
and the definition of $\ell$ give a uniform lower bound
$\ell\geqslant c>0$. Hence every parent chain terminates at a root.

We now count the cores in $\mathcal C_R$. If $N_R\leqslant2$, the desired
estimate is immediate after enlarging the constant. Assume therefore that
$N_R\geqslant3$.

We next select an exceptional subset of $\mathcal C_R$. If a unique core
attains the maximal natural length in $\mathcal C_R$, include that core in the
exceptional set. If at least two distinct cores attain the maximum, no
exception is needed for this reason: each maximal good core can be compared
with another maximal core in Lemma~\ref{lem:crit-scale-drop} and therefore has
a parent. In addition, include the unique good root if it exists and belongs
to $\mathcal C_R$. Denote the resulting exceptional set by $\mathcal X$. Then
\begin{align}
\#\mathcal X\leqslant2.
\label{eq:crit-exceptional-cores}
\end{align}

Let $i\in\mathcal C_R\setminus\mathcal X$ be good. Since $i$ is not the
chosen maximal-scale core, there exists another core
$j\in\mathcal C_R$, $j\neq i$, such that
\begin{align}
\ell_i\leqslant\ell_j.
\end{align}
Both centers belong to $B_{(1+103\Lambda)R}$, and hence
\begin{align}
|x_i-x_j|
\leqslant
2(1+103\Lambda)R.
\label{eq:crit-two-relevant-distance}
\end{align}
Lemma~\ref{lem:crit-scale-drop} produces a core of length at most $\ell_i/4$
and at distance at most $|x_i-x_j|/20$ from $x_i$. Therefore
$\mathfrak p(i)$ exists, and its closest-point definition gives
\begin{align}
e_i
\leqslant
\frac1{20}|x_i-x_j|
\leqslant
\frac1{10}(1+103\Lambda)R.
\label{eq:crit-initial-parent-edge}
\end{align}

Starting from $i$, follow its parent chain until the first bad core is
reached. Such a bad core must exist. Indeed, the chain is finite and ends
at a root; if every core on the chain were good, then the terminal root
would be a good root with a child, which was ruled out above. Denote the
first bad ancestor of a core $i$ by $\mathfrak b(i)$, and set
$b\coloneqq\mathfrak b(i)$ for the core under consideration.

If
\begin{align}
i=i_0,\quad i_1=\mathfrak p(i_0),\quad\ldots,\quad i_m=b,
\end{align}
then $i_0,\ldots,i_{m-1}$ are good. Applying
\eqref{eq:crit-parent-edge-drop} successively gives
\begin{align}
e_{i_s}
\leqslant
20^{-s}e_i,
\qquad
0\leqslant s\leqslant m-1.
\end{align}
Consequently,
\begin{align}
|x_i-x_b|
&\leqslant
e_i+\frac1{20}e_i+\frac1{20^2}e_i+\cdots
\nonumber\\
&\leqslant
\frac{20}{19}e_i.
\label{eq:crit-descendant-ancestor-distance}
\end{align}
Combining this with \eqref{eq:crit-initial-parent-edge} yields
\begin{align}
|x_i-x_b|
<
(1+103\Lambda)R.
\end{align}
Since $|x_i|\leqslant(1+103\Lambda)R$, we obtain
\begin{align}
|x_b|
\leqslant
2(1+103\Lambda)R.
\label{eq:crit-bad-ancestor-location}
\end{align}
The natural lengths decrease along the parent chain, so
$\ell_b\leqslant\ell_i\leqslant2R$. Hence
\begin{align}
b\in\mathcal C_R^{\mathrm{bad}}.
\label{eq:crit-bad-ancestor-in-set}
\end{align}
Thus every relevant good core outside $\mathcal X$ is assigned to a bad
ancestor belonging to $\mathcal C_R^{\mathrm{bad}}$.

It remains to bound the number of relevant good cores assigned to a fixed
bad ancestor $b$. For such a core $i$, retain the first parent-edge length
$e_i$. We divide these cores into dyadic edge-length classes
\begin{align}
\mathcal I_b(r)
\coloneqq
\left\{
i\in\mathcal C_R\setminus\mathcal X:
i\text{ is good},\ \mathfrak b(i)=b,\ r\leqslant e_i<2r
\right\}.
\label{eq:crit-edge-layer}
\end{align}

We claim that each class $\mathcal I_b(r)$ contains at most one core.
Suppose, to the contrary, that $i,j\in\mathcal I_b(r)$ are distinct, and
assume without loss of generality that $\ell_i\leqslant\ell_j$. Since $i$
is good, Lemma~\ref{lem:crit-scale-drop} gives a core $k$ satisfying
\begin{align}
\ell_k
\leqslant
\frac14\ell_i,
\qquad
|x_i-x_k|
\leqslant
\frac1{20}|x_i-x_j|.
\end{align}
By the minimality of the parent $\mathfrak p(i)$,
\begin{align}
e_i
\leqslant
\frac1{20}|x_i-x_j|.
\end{align}
Since $e_i\geqslant r$, we obtain
\begin{align}
|x_i-x_j|
\geqslant
20r.
\label{eq:crit-layer-separation}
\end{align}

On the other hand, since $\mathfrak b(i)=\mathfrak b(j)=b$, the
ancestor-distance estimate
\eqref{eq:crit-descendant-ancestor-distance} and the bounds
$e_i,e_j<2r$ give
\begin{align}
|x_i-x_b|
<
\frac{40}{19}r,
\qquad
|x_j-x_b|
<
\frac{40}{19}r.
\end{align}
Hence
\begin{align*}
|x_i-x_j|
&\leqslant
|x_i-x_b|+|x_j-x_b|\\
&<
\frac{80}{19}r
<
20r,
\end{align*}
contradicting \eqref{eq:crit-layer-separation}. Therefore
\begin{align}
\#\mathcal I_b(r)\leqslant1.
\label{eq:crit-one-layer-count}
\end{align}

For each fixed bad core $b\in\mathcal C_R^{\mathrm{bad}}$, decompose the
family of relevant good cores assigned to $b$ into the disjoint dyadic classes
$\mathcal I_b(r)$. We now show that at most $C\log(e+R)$ of these classes are
nonempty. Indeed, the global lower bound $\ell\geqslant c>0$ and the core
separation \eqref{eq:crit-core-separation} imply
\begin{align}
e_i
=
|x_i-x_{\mathfrak p(i)}|
\geqslant
\Lambda\bigl(\ell_i+\ell_{\mathfrak p(i)}\bigr)
\geqslant
c_0>0,
\label{eq:crit-edge-lower}
\end{align}
whereas \eqref{eq:crit-initial-parent-edge} gives
\begin{align}
e_i\leqslant CR.
\label{eq:crit-edge-upper}
\end{align}
Hence the interval $[c_0,CR]$ contains at most $C\log(e+R)$ dyadic intervals
of the form $[r,2r)$. Therefore, for each fixed
$b\in\mathcal C_R^{\mathrm{bad}}$, there are at most $C\log(e+R)$ nonempty,
pairwise disjoint classes $\mathcal I_b(r)$. Since
\eqref{eq:crit-one-layer-count} gives
\begin{align}
\#\mathcal I_b(r)\leqslant1,
\end{align}
each bad core $b$ is assigned at most $C\log(e+R)$ relevant good cores.

Every bad core already belonging to $\mathcal C_R$ is also contained in
$\mathcal C_R^{\mathrm{bad}}$. Adding these cores, the at most two exceptional
cores in $\mathcal X$, and the good cores assigned to the bad ancestors,
we obtain
\begin{align}
N_R
&\leqslant
2+N_R^{\mathrm{bad}}
+C\log(e+R)N_R^{\mathrm{bad}}
\nonumber\\
&\leqslant
C\left\{
1+\log(e+R)N_R^{\mathrm{bad}}
\right\}.
\label{eq:crit-total-core-count}
\end{align}
This proves the first estimate.

We now prove the estimate for $N_R^{\mathrm{bad}}$. Choose $C_0>1$ so large
that
\begin{align}
B_{\ell_i}(x_i)\subset B_{C_0R}
\qquad
\text{for every }i\in\mathcal C_R^{\mathrm{bad}}.
\label{eq:crit-bad-balls-large-ball}
\end{align}
Such a constant depends only on $\Lambda$: indeed,
$|x_i|\leqslant2(1+103\Lambda)R$ and $\ell_i\leqslant2R$.

Since $i$ is bad, the definition \eqref{eq:crit-good-core} gives
\begin{align}
\ell_i
\int_{B_{\ell_i}(x_i)}
\mathcal D\,\mathrm dx
\geqslant
\delta_*,
\end{align}
and hence
\begin{align}
\frac1{\ell_i}
\leqslant
\frac1{\delta_*}
\int_{B_{\ell_i}(x_i)}
\mathcal D\,\mathrm dx.
\label{eq:crit-bad-cost}
\end{align}
The balls $B_{\ell_i}(x_i)$ are pairwise disjoint because the larger balls
$B_{\Lambda\ell_i}(x_i)$ are pairwise disjoint by
\eqref{eq:crit-core-separation}. Summing \eqref{eq:crit-bad-cost} over
$i\in\mathcal C_R^{\mathrm{bad}}$ and using
\eqref{eq:crit-bad-balls-large-ball}, we obtain
\begin{align}
\sum_{i\in\mathcal C_R^{\mathrm{bad}}}
\ell_i^{-1}
&\leqslant
C
\sum_{i\in\mathcal C_R^{\mathrm{bad}}}
\int_{B_{\ell_i}(x_i)}
\mathcal D\,\mathrm dx
\nonumber\\
&\leqslant
C \mathfrak D(C_0R).
\label{eq:crit-bad-inverse-length}
\end{align}

The same disjointness gives the volume bound
\begin{align}
\sum_{i\in\mathcal C_R^{\mathrm{bad}}}
\ell_i^n
\leqslant
C R^n.
\label{eq:crit-bad-volume}
\end{align}
Indeed, the disjoint balls $B_{\ell_i}(x_i)$ are all contained in
$B_{C_0R}$.

Finally, write
\begin{align}
N_R^{\mathrm{bad}}
&=
\sum_{i\in\mathcal C_R^{\mathrm{bad}}}1
\nonumber\\
&=
\sum_{i\in\mathcal C_R^{\mathrm{bad}}}
\left(\ell_i^{-1}\right)^{n/(n+1)}
\left(\ell_i^n\right)^{1/(n+1)}.
\end{align}
H\"older's inequality with conjugate exponents
$(n+1)/n$ and $n+1$ yields
\begin{align}
N_R^{\mathrm{bad}}
&\leqslant
\left(
\sum_{i\in\mathcal C_R^{\mathrm{bad}}}\ell_i^{-1}
\right)^{n/(n+1)}
\left(
\sum_{i\in\mathcal C_R^{\mathrm{bad}}}\ell_i^n
\right)^{1/(n+1)}
\nonumber\\
&\leqslant
C
\mathfrak D(C_0R)^{n/(n+1)}
R^{n/(n+1)}
\nonumber\\
&=
C\bigl[R \mathfrak D(C_0R)\bigr]^{n/(n+1)}.
\label{eq:crit-bad-count}
\end{align}
This proves the second estimate and completes the proof.
\end{proof}

\subsection{A single integral feedback and completion of the proof}

We now combine the core-counting estimate with the integral estimates above.
Recall the quantity $\mathcal U$ from \eqref{eq:liouville-U}.
We split $B_R$ according to whether the natural length $\ell(x)$ is larger
or smaller than $R$.

On the set $\{\ell\geqslant R\}$, the identity
$\ell=u^{-(p-1)/2}$ gives
\[
u^{1+\beta}
=
\ell^{-2(1+\beta)/(p-1)}
\leqslant
R^{-2(1+\beta)/(p-1)},
\]
because $1+\beta>0$. Hence, by the source bound
\eqref{eq:crit-source-bounds},
\begin{align}
\int_{B_R\cap\{\ell\geqslant R\}}
u^{p+1+\beta}\,\mathrm dx
&\leqslant
R^{-2(1+\beta)/(p-1)}
\int_{B_R}u^p\,\mathrm dx
\nonumber\\
&\leqslant
C
R^{\,n-\frac{2(p+1+\beta)}{p-1}}.
\label{eq:crit-low-set-pre}
\end{align}
Since
\begin{align}
\frac{2(p+1+\beta)}{p-1}\geqslant n-1,
\label{eq:crit-exponent-lower}
\end{align}
we obtain
\begin{align}
\int_{B_R\cap\{\ell\geqslant R\}}
u^{p+1+\beta}\,\mathrm dx
\leqslant
CR.
\label{eq:crit-low-set}
\end{align}

We next estimate the complementary region
$B_R\cap\{\ell<R\}$. Fix
$x\in B_R$ with $\ell(x)<R$, and let $y$ be a minimizer in the definition
of $\underline\ell(x)$. Then
\begin{align}
\underline\ell(x)
=
\ell(y)+\frac{|x-y|}{100\Lambda}
\leqslant
\ell(x)
<
R.
\label{eq:crit-contact-minimizer-high}
\end{align}
In particular,
\begin{align}
\ell(y)<R,
\qquad
|x-y|<100\Lambda R.
\label{eq:crit-contact-high-bounds}
\end{align}

By the covering property \eqref{eq:crit-core-cover}, choose a core
$i\in\mathcal C$ such that
\begin{align}
\ell_i
\leqslant
2\ell(y),
\qquad
|x_i-y|
\leqslant
3\Lambda\ell(y).
\label{eq:crit-cover-high-point}
\end{align}
Then
\begin{align}
\ell_i<2R.
\label{eq:crit-core-length-R}
\end{align}
Moreover, since $x\in B_R$, \eqref{eq:crit-contact-high-bounds} and
\eqref{eq:crit-cover-high-point} give
\begin{align}
|x_i|
&\leqslant
|x_i-y|+|y-x|+|x|
\nonumber\\
&<
3\Lambda R+100\Lambda R+R
\nonumber\\
&=
(1+103\Lambda)R.
\label{eq:crit-core-location-R}
\end{align}
Thus $i\in\mathcal C_R$.

The same estimates give a quantitative comparison between the natural
length at $x$ and the selected core. Indeed,
\begin{align}
\ell_i+\frac{|x-x_i|}{100\Lambda}
&\leqslant
2\ell(y)
+
\frac{|x-y|}{100\Lambda}
+
\frac{|y-x_i|}{100\Lambda}
\nonumber\\
&\leqslant
\frac{203}{100}\ell(y)
+
\frac{|x-y|}{100\Lambda}
\nonumber\\
&\leqslant
\frac{203}{100}
\left(
\ell(y)+\frac{|x-y|}{100\Lambda}
\right)
\nonumber\\
&=
\frac{203}{100}\underline\ell(x)
\nonumber\\
&\leqslant
\frac{203}{100}\ell(x).
\label{eq:crit-core-controls-length}
\end{align}

Set
\begin{align}
q_*
\coloneqq
\frac{2(p+1+\beta)}{p-1}.
\label{eq:crit-q}
\end{align}
Since $\ell=u^{-(p-1)/2}$,
\begin{align}
u^{p+1+\beta}
=
\ell^{-q_*}.
\label{eq:crit-u-length}
\end{align}
By \eqref{eq:crit-exponent-lower}, $q_*\geqslant n-1$. Moreover,
$\ell_i\geqslant c>0$ by the global height bound, and therefore
\[
\ell_i+\frac{|x-x_i|}{100\Lambda}\geqslant c.
\]
Using \eqref{eq:crit-core-controls-length}, we obtain
\begin{align}
u(x)^{p+1+\beta}
&=
\ell(x)^{-q_*}
\nonumber\\
&\leqslant
C'
\left(
\ell_i+\frac{|x-x_i|}{100\Lambda}
\right)^{-q_*}
\nonumber\\
&\leqslant
C
\left(
\ell_i+\frac{|x-x_i|}{100\Lambda}
\right)^{-(n-1)}.
\label{eq:crit-pointwise-core-kernel}
\end{align}
Since for every $x\in B_R\cap\{\ell<R\}$ at least one such
$i\in\mathcal C_R$ exists, we may sum over all cores in $\mathcal C_R$ and
obtain
\begin{align}
u(x)^{p+1+\beta}\mathbf 1_{\{\ell(x)<R\}}
\leqslant
C
\sum_{i\in\mathcal C_R}
\left(
\ell_i+\frac{|x-x_i|}{100\Lambda}
\right)^{-(n-1)},
\qquad x\in B_R.
\label{eq:crit-high-pointwise}
\end{align}

We next integrate each kernel. If $x\in B_R$ and $i\in\mathcal C_R$, then
\begin{align}
|x-x_i|
\leqslant
|x|+|x_i|
\leqslant
C R.
\end{align}
Hence, using polar coordinates centered at $x_i$,
\begin{align}
\int_{B_R}
\left(
\ell_i+\frac{|x-x_i|}{100\Lambda}
\right)^{-(n-1)}
\,\mathrm dx
&\leqslant
C
\int_0^{CR}
\frac{t^{n-1}}
{\left(\ell_i+t/(100\Lambda)\right)^{n-1}}
\,\mathrm dt
\nonumber\\
&\leqslant
CR.
\label{eq:crit-kernel-integral}
\end{align}

Integrating \eqref{eq:crit-high-pointwise} over $B_R$ and using
\eqref{eq:crit-kernel-integral}, we obtain
\begin{align}
\int_{B_R\cap\{\ell<R\}}
u^{p+1+\beta}\,\mathrm dx
\leqslant
C R N_R.
\label{eq:crit-high-set}
\end{align}
Combining \eqref{eq:crit-low-set} and \eqref{eq:crit-high-set} gives
\begin{align}
\mathcal U(R)
\leqslant
CR(1+N_R).
\label{eq:crit-U-before-count}
\end{align}

Lemma~\ref{lem:crit-count} gives
\begin{align}
N_R
&\leqslant
C\left\{
1+\log(e+R)N_R^{\mathrm{bad}}
\right\},
\label{eq:crit-count-use-1}\\
N_R^{\mathrm{bad}}
&\leqslant
C
\left[
R\mathfrak D(C_0R)
\right]^{n/(n+1)}.
\label{eq:crit-count-use-2}
\end{align}
Substituting these estimates into
\eqref{eq:crit-U-before-count} and absorbing the lower-order $CR$ terms,
we obtain, after enlarging $C_0$ if necessary,
\begin{align}
\mathcal U(R)
\leqslant
CR
+
CR\log(e+R)
\left[
R\mathfrak D(C_0R)
\right]^{n/(n+1)}.
\label{eq:crit-U-feedback}
\end{align}

We now convert this estimate for $\mathcal U$ into one for
$\mathfrak D$. By \eqref{eq:liouville-cutoff},
\begin{align}
\mathfrak D(R)
\leqslant
CR^{-2}\mathcal E(2R).
\label{eq:crit-D-from-E}
\end{align}
On the other hand, \eqref{eq:crit-test} gives
\begin{align}
\mathcal E(2R)
\leqslant
C\mathcal U(4R)+CR.
\label{eq:crit-E-from-U}
\end{align}
Combining \eqref{eq:crit-D-from-E} and
\eqref{eq:crit-E-from-U}, and then applying
\eqref{eq:crit-U-feedback} at radius $4R$, we obtain
\begin{align}
\mathfrak D(R)
&\leqslant
\frac{C}{R^2}
\left\{
\mathcal U(4R)+R
\right\}
\nonumber\\
&\leqslant
\frac{C}{R}
+
\frac{C\log(e+R)}{R}
\left[
R\mathfrak D(C_0R)
\right]^{n/(n+1)}.
\label{eq:crit-feedback-pre}
\end{align}
Here and below $C_0$ is enlarged by a fixed factor whenever necessary.
Since
\begin{align}
\frac1R
\left[
R\mathfrak D(C_0R)
\right]^{n/(n+1)}
=
\frac{
\mathfrak D(C_0R)^{n/(n+1)}
}{
R^{1/(n+1)}
},
\end{align}
we arrive at the common feedback inequality
\begin{align}
\mathfrak D(R)
\leqslant
\frac{C}{R}
+
\frac{C\log(e+R)}{R^{1/(n+1)}}
\mathfrak D(C_0R)^{n/(n+1)}.
\label{eq:crit-feedback}
\end{align}

It remains to close this inequality. First note that $\mathfrak D$ has at
most polynomial growth. Indeed, \eqref{eq:crit-polynomial},
\eqref{eq:crit-test}, and \eqref{eq:liouville-cutoff} give
\begin{align}
\mathfrak D(R)
\leqslant
C(1+R)^M
\label{eq:crit-D-polynomial-growth}
\end{align}
for some fixed $M>0$.

Assume that $\mathfrak D\not\equiv0$. Since $\mathfrak D$ is nonnegative
and nondecreasing, there exists $R_1$ with $\mathfrak D(R_1)>0$. We claim
that there exist $R_j\to\infty$ and a constant $A>1$, independent of $j$,
such that
\begin{align}
\mathfrak D(C_0R_j)
\leqslant
A\mathfrak D(R_j).
\label{eq:crit-radius-selection}
\end{align}
This is the standard polynomial-growth radius-selection argument of
\cite[Section~3, Step~5, (3.28), p.~1424]{Souplet2009}. For completeness,
choose $A>C_0^M$. If \eqref{eq:crit-radius-selection} failed for all
sufficiently large $R$, then
\[
\mathfrak D(C_0R)>A\mathfrak D(R)
\]
for all such $R$. Iterating at the radii $R,C_0R,C_0^2R,\ldots$ would give
\begin{align}
\mathfrak D(C_0^kR)
>
A^k\mathfrak D(R),
\end{align}
whereas \eqref{eq:crit-D-polynomial-growth} gives
\begin{align}
\mathfrak D(C_0^kR)
\leqslant
C C_0^{kM}R^M.
\end{align}
Since $A>C_0^M$, these two inequalities imply that, for every
$k\geqslant1$,
\begin{align}
\mathfrak D(R)<C\left(\frac{C_0^M}{A}\right)^kR^M.
\end{align}
Letting $k\to\infty$ yields $\mathfrak D(R)=0$, contrary to the assumption.
Thus there exists a sequence $R_j\to\infty$ along which
\eqref{eq:crit-radius-selection} holds.

Substituting \eqref{eq:crit-radius-selection} into
\eqref{eq:crit-feedback}, we obtain
\begin{align}
\mathfrak D(R_j)
\leqslant
\frac{C}{R_j}
+
\frac{C\log(e+R_j)}{R_j^{1/(n+1)}}
\mathfrak D(R_j)^{n/(n+1)}.
\label{eq:crit-feedback-selected}
\end{align}
Apply Young's inequality in the form
\begin{align}
ab^{n/(n+1)}
\leqslant
\frac12 b
+
C a^{n+1},
\qquad
a,b\geqslant0.
\label{eq:crit-young}
\end{align}
With
\[
a=
\frac{C\log(e+R_j)}{R_j^{1/(n+1)}},
\qquad
b=\mathfrak D(R_j),
\]
we deduce from \eqref{eq:crit-feedback-selected} that
\begin{align}
\mathfrak D(R_j)
\leqslant
\frac12\mathfrak D(R_j)
+
\frac{C\left\{1+\log^{n+1}(e+R_j)\right\}}{R_j}.
\label{eq:crit-feedback-final}
\end{align}
Therefore
\begin{align}
\mathfrak D(R_j)
\leqslant
\frac{
C\left\{1+\log^{n+1}(e+R_j)\right\}
}{R_j}
\longrightarrow
0.
\end{align}
Since $\mathfrak D$ is nonnegative and nondecreasing, it follows that
$\mathfrak D\equiv0$. The common criterion
\eqref{eq:liouville-criterion} then gives $u\equiv0$. Together with the
elementary cases treated in the first subsection, this completes the proof of
Theorem~\ref{thm:main}.

\subsection{Geometric consequence and the one-dimensional case}

\begin{proof}[Proof of Corollary~\ref{cor:geometric-halfspace}]
Let $\pi_P:\Sigma\to P$ be the orthogonal projection. For
$V\in T\Sigma$, the decomposition
\[
\mathrm dF(V)=\mathrm d\pi_P(V)+\mathrm d\tau(V)T
\]
and $\langle T,T\rangle_{\mathbb R^{n,1}}=-1$ give
\[
|\mathrm d\pi_P(V)|_P^2
=|V|_\Sigma^2+|\mathrm d\tau(V)|^2
\geqslant|V|_\Sigma^2.
\]
Thus $\pi_P$ is a local diffeomorphism whose differential does not
decrease lengths.

We include the covering argument. Fix $q_0\in\Sigma$, and let
$c:[0,1]\to P$ be a piecewise $C^1$ path beginning at $\pi_P(q_0)$. Suppose
that its maximal lift $\widetilde c$ starting at $q_0$ is defined on
$[0,b)$ with $b\leqslant1$. The preceding inequality gives
\[
|\widetilde c'(t)|_\Sigma\leqslant|c'(t)|_P.
\]
The lift has finite length on $[0,b)$ and is Cauchy as $t\uparrow b$.
Completeness gives a limit point, so the lift extends to $b$. If $b<1$,
the local diffeomorphism property extends it beyond $b$, contradicting
maximality. Hence every such path lifts to $[0,1]$. It follows that $\pi_P$ is surjective and is
a covering map. Since $P\simeq\mathbb R^n$ is simply connected and $\Sigma$
is connected, $\pi_P$ is a global diffeomorphism.

Consequently, $F(\Sigma)$ is the entire spacelike graph over $P$ of the
nonnegative function $u\coloneqq\tau$. After a Lorentz isometry sending $(P,T)$ to
$(\mathbb R^n\times\{0\},e_{n+1})$, the curvature equation becomes
\eqref{eq:main-equation}. Theorem~\ref{thm:main} gives $u\equiv0$, and hence
$F(\Sigma)=P$.
\end{proof}

\begin{remark}[The one-dimensional case]\label{rem:one-dimensional}
For $n=1$ and $p\geqslant1$, every nonnegative entire strictly spacelike
$C^2$ solution is zero. Indeed, if a solution is nontrivial, uniqueness for
the associated first-order system implies that it is positive. Set
\[
X\coloneqq\frac{u'}{\sqrt{1-(u')^2}}.
\]
Then $X'=-u^p<0$. If $X$ is negative at some point, it remains bounded above
by a negative constant to the right, so $u'$ does as well and $u$ must cross
zero. If $X(x_0)\geqslant0$, then while $X\geqslant0$ one has
$u\geqslant u(x_0)>0$ and hence
$X'\leqslant-u(x_0)^p$; thus $X$ becomes negative in finite time, reducing to
the first alternative. Both possibilities are impossible.
\end{remark}

\raggedbottom
\begin{samepage}
\section*{Acknowledgments}
Xi-Nan Ma and Tian Wu were supported by the National Key Research and Development Program of China (Grant No. 2025YFA1017600). Xi-Nan Ma was supported by the National Natural Science Foundation of China (Grant No. 12141105). Tian Wu was supported by the Fundamental Research Funds for the Central Universities (Grant No. WK0010250106).

\section*{Declaration of generative AI and AI-assisted technologies in the manuscript preparation process}
During the preparation of this work, the authors used ChatGPT (OpenAI) to
improve the language and organization of the manuscript and to assist in
checking the consistency of the mathematical exposition. After using this
tool, the authors independently reviewed, edited, and verified the content as
needed and take full responsibility for the content of the published article.
\end{samepage}


\begin{thebibliography}{99}\small

\bibitem{Azzollini2014}
A. Azzollini,
\emph{Ground state solution for a problem with mean curvature operator in
Minkowski space},
J. Funct. Anal. 266 (2014), 2086--2095.

\bibitem{Azzollini2016}
A. Azzollini,
\emph{On a prescribed mean curvature equation in Lorentz--Minkowski space},
J. Math. Pures Appl. (9) 106 (2016), 1122--1140.

\bibitem{BartnikSimon1982}
R. Bartnik and L. Simon,
\emph{Spacelike hypersurfaces with prescribed boundary values and mean curvature},
Comm. Math. Phys. 87 (1982), 131--152.

\bibitem{BonheureDaveniaPomponio2016}
D. Bonheure, P. d'Avenia, and A. Pomponio,
\emph{On the electrostatic Born--Infeld equation with extended charges},
Comm. Math. Phys. 346 (2016), 877--906.

\bibitem{BonheureDeCosterDerlet2012}
D. Bonheure, C. De Coster, and A. Derlet,
\emph{Infinitely many radial solutions of a mean curvature equation in
Lorentz--Minkowski space},
Rend. Istit. Mat. Univ. Trieste 44 (2012), 259--284.

\bibitem{BornInfeld1934}
M. Born and L. Infeld,
\emph{Foundations of the new field theory},
Proc. R. Soc. Lond. Ser. A 144 (1934), 425--451.

\bibitem{ByeonIkomaMalchiodiMari2024}
J. Byeon, N. Ikoma, A. Malchiodi, and L. Mari,
\emph{Existence and regularity for prescribed Lorentzian mean curvature
hypersurfaces, and the Born--Infeld model},
Ann. PDE 10 (2024), Paper No.~4, 86 pp.

\bibitem{ByeonIkomaMalchiodiMari2026}
J. Byeon, N. Ikoma, A. Malchiodi, and L. Mari,
\emph{Compactness via monotonicity in nonsmooth critical point theory,
with application to Born--Infeld type equations},
J. Funct. Anal. 290 (2026), no.~11, Paper No.~111438, 77 pp.

\bibitem{CaffarelliGidasSpruck1989}
L. A. Caffarelli, B. Gidas, and J. Spruck,
\emph{Asymptotic symmetry and local behavior of semilinear elliptic equations
with critical Sobolev growth},
Comm. Pure Appl. Math. 42 (1989), 271--297.

\bibitem{ChengYau1976}
S.-Y. Cheng and S.-T. Yau,
\emph{Maximal space-like hypersurfaces in the Lorentz--Minkowski spaces},
Ann.\ of Math. 104 (1976), 407--419.

\bibitem{GidasNiNirenberg1979}
B. Gidas, W.-M. Ni, and L. Nirenberg,
\emph{Symmetry and related properties via the maximum principle},
Comm. Math. Phys. 68 (1979), 209--243.

\bibitem{GidasSpruck1981}
B. Gidas and J. Spruck,
\emph{Global and local behavior of positive solutions of nonlinear elliptic equations},
Comm. Pure Appl. Math. 34 (1981), 525--598.

\bibitem{GT}
D. Gilbarg and N. S. Trudinger,
\emph{Elliptic Partial Differential Equations of Second Order}, 2nd ed.,
Classics in Mathematics, Springer, Berlin, 2001.

\bibitem{LiLiuMa2026}
J. Li, Y. Liu, and X.-N. Ma,
\emph{Classification of positive entire solutions of the critical CR Yamabe
equation on the Heisenberg group},
author-uploaded preprint, August 12, 2026,
\href{https://doi.org/10.13140/RG.2.2.13416.43527}
{doi:10.13140/RG.2.2.13416.43527}.

\bibitem{LSW1963}
W. Littman, G. Stampacchia, and H. F. Weinberger,
\emph{Regular points for elliptic equations with discontinuous coefficients},
Ann. Scuola Norm. Sup. Pisa (3) 17 (1963), 43--77.

\bibitem{MaWuInvariant2024}
X.-N. Ma and T. Wu,
\emph{The application of the invariant tensor technique in the classification
of solutions to semilinear elliptic and sub-elliptic partial differential
equations},
Sci. Sin. Math. 54 (2024), 1627--1648 (in Chinese).

\bibitem{MaWu2026}
X.-N. Ma and W. Wu,
\emph{Liouville theorem for elliptic equations with a source reaction term
involving the product of the function and its gradient in $\mathbb R^n$},
Bull. Sci. Math. 206 (2026), 103747.

\bibitem{Obata1962}
M. Obata,
\emph{Certain conditions for a Riemannian manifold to be isometric with a sphere},
J. Math. Soc. Japan 14 (1962), 333--340.

\bibitem{PolacikQuittnerSouplet2007}
P. Pol\'a\v cik, P. Quittner, and Ph. Souplet,
\emph{Singularity and decay estimates in superlinear problems via
Liouville-type theorems. I. Elliptic equations and systems},
Duke Math. J. 139 (2007), 555--579.

\bibitem{SerrinZou2002}
J. Serrin and H. Zou,
\emph{Cauchy--Liouville and universal boundedness theorems for quasilinear
elliptic equations and inequalities},
Acta Math. 189 (2002), 79--142.

\bibitem{Souplet2009}
Ph. Souplet,
\emph{The proof of the Lane--Emden conjecture in four space dimensions},
Adv. Math. 221 (2009), 1409--1427.

\bibitem{Treibergs1982}
A. E. Treibergs,
\emph{Entire spacelike hypersurfaces of constant mean curvature in Minkowski space},
Invent. Math. 66 (1982), 39--56.

\bibitem{Zhang2026}
Y. R.-Y. Zhang,
\emph{Critical $p$-Laplace equations with monotone coefficients:
Liouville classification and a Schoen-type Harnack inequality},
\href{https://arxiv.org/abs/2608.09113v2}{arXiv:2608.09113v2}, 2026.

\end{thebibliography}
\end{document}